\documentclass{amsart}
\usepackage{graphicx} 
\usepackage[margin=1in]{geometry}
\usepackage{microtype}
\usepackage{amsfonts, amssymb, amsmath}
\usepackage{enumitem}
\usepackage{cite}
\usepackage{xcolor}
\usepackage{hyperref}

\hypersetup{
    colorlinks=true,
    linkcolor=blue,
    citecolor=blue,
    urlcolor=blue
}

\hypersetup{
  pdftitle={123-Avoiding Permutations with an Adjacency Constraint},
  pdfauthor={Nathaniel Nadler},
  pdfsubject={Pattern avoidance and bounded adjacent differences},
  pdfkeywords={
    pattern avoidance,
    bounded-gap permutations,
    generating trees,
    transfer matrices,
    rational generating functions,
    Hamiltonian paths
  }
}

\usepackage{amsfonts,amsmath, amsthm}
\title{123-Avoiding Permutations with an Adjacency Constraint}
\author{Nathaniel Nadler}
\thanks{I extend my deepest thanks to my advisor, Dr. Katarzyna Kowal.}
\address{Ramapo College of New Jersey}
\email{nnadler@ramapo.edu}
\date{September 2026}
\newtheorem{theorem}{Theorem}
\newtheorem{lemma}{Lemma}
\newtheorem{prop}{Proposition}
\newtheorem{cor}{Corollary}
\newtheorem{definition}{Definition}
\newtheorem{remark}{Remark}

\begin{document}

\begin{abstract}
   We study permutations $\pi \in S_n$ that avoid $123$ while satisfying the adjacency constraint $\lvert \pi_{i+1}-\pi_i\rvert \le m$ for some $m\in\mathbb{Z}^{+}$. We first show that every admissible permutation is localized within vertical distance $m$ of the reverse diagonal, revealing a strong global restriction imposed by the interaction between pattern avoidance and bounded adjacency. For every fixed $m$, we then construct an exact finite-state description of the family, with precisely $2^{m+1}-m-1$ realizable states. This yields an explicit transfer-matrix enumeration and, in particular, a rational generating function for every fixed $m$; consequently, the enumeration sequence satisfies an eventual constant-coefficient linear recurrence. We further identify the exponential growth rate with the spectral radius $\rho_m$ of the corresponding transfer matrix and show that $\rho_m<4$, is nondecreasing in $m$, and converges to the Catalan growth constant $4$ as $m\to\infty$. Finally, we show that the bounded-adjacency constraint breaks the ordinary length-three Wilf equivalence, and we develop algorithmic and graph-theoretic interpretations of the family, including a correspondence with constrained Hamiltonian paths in powers of the path graph. 
\end{abstract}
\maketitle

\section{Introduction}

Pattern avoidance studies permutations by asking which smaller orderings can or cannot appear inside them. One of the most classical examples is $123$-avoidance. A permutation avoids $123$ when it contains no three entries, appearing from left to right, whose values are also increasing. These permutations are well understood: they form a Catalan family and are counted by the Catalan numbers~\cite{SimionSchmidt1985}.

In this paper, we ask what happens when $123$-avoidance is combined with a second condition of a very different kind. Rather than restricting the relative order of entries that may be far apart, we also restrict how much the permutation may change from one position to the next. Restrictions on gaps between neighboring entries have previously been studied in their own right. Avgustinovich and Kitaev studied uniquely \(k\)-determined permutations, which are equivalent, after inversion and a parameter shift, to non-anchored bounded-gap permutations. Their finite-prohibition description makes transfer-matrix enumeration possible~\cite{AvgustinovichKitaev2008,GillespieMonksMonks2020}. Gillespie, Monks, and Monks later studied the anchored bounded-gap problem and proved rationality of its generating function for every fixed gap bound~\cite{GillespieMonksMonks2020}.

More recently, the interaction between bounded adjacency and pattern avoidance was considered for \(132\)-avoidance~\cite{Nadler2026}. Mayama and Akita subsequently developed a finite-state treatment of the adjacency-constrained \(132\) problem for every fixed \(m\), using a two-sided state system based on endpoint deficiencies~\cite{MayamaAkita2026}. The \(123\)-avoiding case considered here has a different insertion structure: legality of inserting a new maximum is controlled by an initial decreasing segment. This leads to the active-profile description developed below.

We fix $m\in\mathbb{Z}_{>0}$, requiring
\[
\lvert \pi_{i+1}-\pi_i\rvert\leq m
\]
for every $1\leq i<n$.

The interest is in the interaction between the two restrictions; avoiding $123$ is a global condition --- the three entries involved may be separated by many positions --- whereas the adjacency bound is local --- it only compares neighboring entries. On its own, neither condition says that a permutation must stay close to one particular ordering. However, when the two are imposed together that is exactly what happens.

Our first main result shows that every admissible permutation stays close to the decreasing permutation. More precisely, if $\pi\in\mathcal{C}_{n,m}$, then
\[
\lvert \pi_i-(n+1-i)\rvert\leq m
\]
for every $1\leq i\leq n$. Since $n+1-i$ is the entry in position $i$ of the decreasing permutation $n,n-1,\ldots,1$, this says that every entry of $\pi$ differs from the corresponding entry of the decreasing permutation by at most $m$. This is the main structural piece behind the paper. Interestingly, the adjacency condition only controls consecutive entries, but once $123$-avoidance is added, the effect becomes global. Specifically, the entire permutation is forced to remain close to the reverse diagonal.

Furthermore, that observation leads to the enumeration. We build permutations by inserting a new maximum and ask where that insertion can be made without violating either condition. Not every part of the permutation matters for this question. For fixed $m$, only a bounded amount of information near the beginning of the permutation affects where the next maximum may be inserted. We record that information using \emph{active profiles}. The profile keeps track of the entries that matter for future insertions, measured by their distance below the current maximum. The key point is that two permutations with the same active profile behave the same way under the next insertion: corresponding legal insertions lead to the same new profiles. The full permutation may be arbitrarily long, but for fixed $m$ there are only finitely many possible profiles.

We characterize those profiles exactly and show that there are
\[
2^{m+1}-m-1
\]
realizable states, including the state in which no further insertion is possible. The possible moves between these states define a finite transfer matrix $T_m$. If $u_{n,m}$ records how many permutations of length $n$ occur in each state, then
\[
u_{n+1,m}=u_{n,m}T_m.
\]
In this way, the original permutation problem becomes a finite matrix recurrence. Let $\lvert C_{n,m}\rvert = a_{n,m}$.
\[
\sum_{n\geq 0}a_{n,m}z^n
\]
is rational, and therefore, the sequence $(a_{n,m})_{n\geq 0}$ eventually satisfies a constant-coefficient linear recurrence.

We then study how this growth rate changes with $m$. For every fixed $m$, we have $\rho_m<4$. Moreover, $\rho_m$ is nondecreasing in $m$ and
\[
\lim_{m\to\infty}\rho_m=4.
\]
Since the ordinary $123$-avoiding permutations have exponential growth rate $4$, this shows exactly what one would expect from relaxing the adjacency condition: every fixed bound reduces the growth rate, but as larger jumps are allowed, the growth approaches that of the full Catalan class.

The adjacency condition also changes the usual symmetry picture for patterns of length three. In the unconstrained setting, all six patterns in $S_3$ are Wilf-equivalent~\cite{SimionSchmidt1985}. Under bounded adjacency, this equivalence breaks: reversal and complement still preserve the condition, while inversion does not in general. Consequently, we obtain two symmetry classes whose enumerations need not agree.

Finally, the adjacency condition has a natural graph interpretation. If $P_n^m$ denotes the $m$th power of the path graph on $[n]$, then a permutation satisfying the adjacency bound is exactly an oriented Hamiltonian path in $P_n^m$. Adding $123$-avoidance restricts the order in which the vertices may be visited. This viewpoint gives another way to understand the class and also leads to algorithmic consequences for counting and generating its elements.

The paper proceeds as follows. We first establish the structural properties forced by the two constraints. We then develop the active-profile description and the finite-state enumeration. After that, we study the resulting growth rates, examine the pattern symmetries that remain under the adjacency condition, and conclude with the graph-theoretic and algorithmic interpretations.

\section{Structural properties and localization}
\label{sec:structure}

We say that
\[
\mathcal C_{n,m}
=
\left\{
\pi\in S_n:
\pi\text{ avoids }123
\text{ and }
|\pi_{i+1}-\pi_i|\le m
\text{ for }1\le i<n
\right\}.
\]

Our eventual goal is to enumerate $\mathcal C_{n,m}$ for fixed $m$ by
building permutations recursively, inserting a new maximum at each step.
For such a recursion to be useful, however, two things must first be true.
Deleting the maximum should take us back to a smaller permutation in the
same class, and the number of ways to insert the next maximum should depend
on only a bounded amount of information.

Neither fact is obvious from the definition. The adjacency condition is
local: it only compares consecutive entries. Pattern avoidance, on the
other hand, is global. The useful feature of this problem is that these two
conditions reinforce one another. Together they force every admissible
permutation to remain close to the decreasing permutation
\[
n,n-1,\ldots,1.
\]
This localization is the main structural fact behind the finite-state
description developed later.

\subsection{Localization around the decreasing permutation}

For the decreasing permutation $\delta_n=n(n-1)\cdots 1$, the entry in
position $i$ is exactly
\[
\delta_n(i)=n+1-i.
\]
The next theorem says that imposing bounded adjacent differences prevents a
$123$-avoiding permutation from moving very far away from this value.

\begin{theorem}[Reverse-diagonal localization]
\label{thm:localization}
For every $\pi\in\mathcal C_{n,m}$ and every $1\le i\le n$,
\[
\left|\pi_i-(n+1-i)\right|\le m.
\]
\end{theorem}

Before proving the theorem, it is useful to understand why such a statement
should be true. Fix an entry $x=\pi_i$. There are $n-x$ values larger than
$x$ somewhere in the permutation. At most $i-1$ of them can occur to the
left of $x$, simply because there are only $i-1$ positions there. Thus, the important question is how many values larger than $x$ can occur to its right.

The adjacency condition alone does not answer this: a permutation could in
principle climb through several jumps of size at most $m$. The
$123$-avoidance condition prevents exactly that behavior. Once the
permutation makes its first move above $x$, that first larger value becomes
a ceiling for all later values larger than $x$. Otherwise $x$, that first
larger value, and a still larger later value would form a $123$ pattern.
Since the first move above $x$ has size at most $m$, only boundedly many
values larger than $x$ can occur to the right.

\begin{proof}
Fix $i$ and write
\[
x=\pi_i.
\]

Suppose first that there is some entry to the right of $x$ that is larger
than $x$, and let $j>i$ be the first index for which
\[
\pi_j>x.
\]
By the choice of $j$, we have
\[
\pi_{j-1}\le x.
\]
Since adjacent entries differ by at most $m$,
\[
\pi_j-\pi_{j-1}\le m,
\]
and thus
\[
\pi_j-x\le m.
\]

Now consider any $k>j$. We cannot have
\[
\pi_k>\pi_j,
\]
because then
\[
\pi_i=x<\pi_j<\pi_k
\]
with
\[
i<j<k,
\]
which would give a $123$ pattern. Thus every value larger than $x$ that
appears to the right of position $i$ belongs to
\[
\{x+1,x+2,\ldots,\pi_j\}.
\]
There are at most $m$ such values.

If no entry larger than $x$ occurs to the right, the same conclusion is
trivial. Hence, in every case,
\[
\#\{k>i:\pi_k>x\}\le m.
\]

There are at most $i-1$ values larger than $x$ to the left of position $i$.
Since exactly $n-x$ values of $[n]$ are larger than $x$, we obtain
\[
n-x
\le
(i-1)+m.
\]
Rearranging gives
\[
x\ge n+1-i-m.
\]

It remains to obtain the opposite inequality. Define the reverse complement of $\pi$ by
\[
\pi^{rc}_k=n+1-\pi_{n+1-k}.
\]
Reverse-complement preserves $123$-avoidance and also preserves absolute
adjacent differences, so
\[
\pi^{rc}\in\mathcal C_{n,m}.
\]
Applying the inequality just proved to $\pi^{rc}$ at position $n+1-i$
gives
\[
n+1-x\ge i-m.
\]
Equivalently,
\[
x\le n+1-i+m.
\]

Combining the two bounds yields
\[
n+1-i-m
\le
\pi_i
\le
n+1-i+m,
\]
which is exactly
\[
\left|\pi_i-(n+1-i)\right|\le m.
\]
\end{proof}

The theorem has a simple geometric interpretation. If the permutation is
drawn as the set of points
\[
(i,\pi_i),\qquad 1\le i\le n,
\]
then the decreasing permutation lies on the line
\[
y=n+1-x.
\]
Theorem~\ref{thm:localization} says that every point \((i,\pi_i)\) lies in the band
\[
|x+y-(n+1)|\le m.
\]
Equivalently, at each fixed position \(i\), the value \(\pi_i\) differs from the corresponding value of the decreasing permutation by at most \(m\).

Thus the adjacency restriction, although defined only between neighboring
entries, has produced a global restriction on the entire permutation. For the recursive construction, the most useful consequence is that the
largest values are forced to occur near the beginning.

To preface the following Corollary, for \(x\in[n]\), let
\[
\operatorname{pos}_{\pi}(x)
\]
be the position occupied by \(x\) in \(\pi\).

\begin{cor}[Localization of the largest values]
\label{cor:large-values}
Let $\pi\in\mathcal C_{n,m}$ and let $0\le r\le n-1$. Then
\[
\left|
\operatorname{pos}_\pi(n-r)-(r+1)
\right|
\le m.
\]
In particular,
\[
\operatorname{pos}_\pi(n-r)\le r+m+1.
\]
\end{cor}

\begin{proof}
Let
\[
i=\operatorname{pos}_\pi(n-r),
\]
so that $\pi_i=n-r$. Applying
Theorem~\ref{thm:localization} gives
\[
\left|
(n-r)-(n+1-i)
\right|
\le m.
\]
The expression inside the absolute value simplifies to
\[
i-(r+1),
\]
and hence
\[
\left|i-(r+1)\right|\le m.
\]
\end{proof}

For example, the maximum $n$ must occur among the first $m+1$ positions, $n-1$ must occur among the first $m+2$ positions, and so on. In the decreasing permutation, these values occur in positions $1,2,\ldots$, respectively; Corollary~\ref{cor:large-values} says that an admissible permutation can move them only a bounded distance away. This is the first indication that the behavior relevant to inserting future maxima should be concentrated near the beginning of the permutation.

\subsection{Deleting the maximum}

Localization tells us that the largest values are confined to a small part of the permutation. To turn this observation into a recursive enumeration, we next need to know that removing the current maximum does not take us out of the class. Pattern avoidance causes no difficulty: deleting an entry cannot create a new $123$ pattern. The only possible problem is the adjacency condition. When the maximum is deleted, its two former neighbors become adjacent, and we must check that this new pair still differs by at most $m$.

\begin{lemma}[Closure under maximum deletion]
\label{lem:max-deletion}
Let
\[
\pi\in\mathcal C_{n+1,m}.
\]
Delete the entry $n+1$. The resulting permutation belongs to
$\mathcal C_{n,m}$.
\end{lemma}

\begin{proof}
Deleting an entry cannot create a new occurrence of $123$, so the resulting
permutation still avoids $123$.

It remains only to check the adjacency condition. If $n+1$ occurs at one
of the two ends of $\pi$, no new adjacency is created, so there is nothing
to prove.

Suppose instead that $\pi$ contains
\[
\cdots,a,n+1,b,\cdots.
\]
Because $\pi\in\mathcal C_{n+1,m}$,
\[
n+1-a\le m
\qquad\text{and}\qquad
n+1-b\le m.
\]
Hence
\[
a,b\in
\{n+1-m,n+2-m,\ldots,n\}.
\]
Any two values in this set differ by at most $m-1$, so
\[
|a-b|\le m-1\le m.
\]
Thus, after deleting $n+1$, the newly adjacent pair $a,b$ still satisfies the adjacency restriction.
\end{proof}

Lemma~\ref{lem:max-deletion} gives every permutation in $\mathcal C_{n+1,m}$ a unique parent in $\mathcal C_{n,m}$: simply delete its maximum. Accordingly, we may reverse the operation and try to construct all permutations of size $n+1$ by inserting $n+1$ into permutations of size $n$. The remaining question is central for the intended recursive description: \emph{At which sites may the new maximum be inserted?}

\subsection{Active insertion sites}

A standard way to enumerate pattern-avoiding permutations recursively is to insert the new maximum into designated active sites, yielding a generating tree~\cite{West1995}.

Let
\[
\pi=\pi_1\pi_2\cdots\pi_n\in\mathcal C_{n,m}.
\]
Write $s_j$ for the site immediately after the first $j$ entries, so that $s_0$ is the site before $\pi_1$. Let $I_j(\pi)$ be the permutation obtained by inserting $n+1$ at $s_j$. We call $s_j$ \emph{active} if
\[
I_j(\pi)\in\mathcal C_{n+1,m}.
\]

There are again two separate conditions to check. First, inserting a new maximum can create a $123$ pattern only in one way: the new maximum can become the final and largest member of the pattern. Thus any increasing pair appearing before the insertion site immediately makes that site illegal.

\begin{lemma}[Pattern condition for insertion]
\label{lem:decreasing-prefix}
For $j\ge 1$, inserting $n+1$ at $s_j$ preserves $123$ avoidance if and only if
\[
\pi_1>\pi_2>\cdots>\pi_j.
\]
\end{lemma}

\begin{proof}
Suppose first that the prefix is not decreasing. Then there exist indices
\[
1\le a<b\le j
\]
such that
\[
\pi_a<\pi_b.
\]
After inserting $n+1$ at $s_j$, these entries occur before the new maximum,
so
\[
\pi_a<\pi_b<n+1
\]
forms a $123$ pattern.

Conversely, suppose
\[
\pi_1>\pi_2>\cdots>\pi_j.
\]
The original permutation already avoids $123$. Thus any new $123$ pattern would have to use the newly inserted value $n+1$. Since $n+1$ is the largest value, it would have to be the third member of such a pattern, which would require an increasing pair among the first $j$ entries. No such pair exists because the prefix is decreasing.
\end{proof}

For the remainder of this subsection, assume that \(n\ge m+1\).

The adjacency condition is even more direct. An old entry $x$ may sit next to the new maximum $n+1$ precisely when
\[
n+1-x\le m,
\]
or equivalently
\[
x\ge n+1-m.
\]
This motivates the following notation:
\[
H_{n,m}
=
\{n+1-m,n+2-m,\ldots,n\}.
\]
We will call the elements of $H_{n,m}$ the \emph{high values}. They are exactly the $m$ entries that are large enough to be adjacent to $n+1$.

Thus an active insertion site must satisfy two simple requirements:

\begin{enumerate}
    \item the entries before the site must form a decreasing prefix;
    \item the entries immediately adjacent to the new maximum must be high.
\end{enumerate}

These two requirements allow us to describe all active sites at once.

Define
\[
d(\pi)
=
\max\left\{
t:
\pi_1>\pi_2>\cdots>\pi_t
\right\},
\]
the length of the initial decreasing run, and define
\[
h(\pi)
=
\max\left\{
t:
\pi_1,\ldots,\pi_t\in H_{n,m}
\right\},
\]
with $h(\pi)=0$ when $\pi_1\notin H_{n,m}$.

For the sake of clarity, let
\[
\operatorname{Act}(\pi)
\]
denote the set of active sites of \(\pi\).

\begin{prop}[Structure of the active sites]
\label{prop:active-sites}
Assume $n\ge m+1$.

If $h(\pi)=0$, then $\pi$ has no active sites. If $h(\pi)\ge1$, then
\[
\operatorname{Act}(\pi)
=
\{s_0,s_1,\ldots,s_r\},
\qquad
r=
\min\{d(\pi),h(\pi)-1\}.
\]
In particular, the active sites always form an initial segment.
\end{prop}

\begin{proof}
Suppose first that \(h(\pi)=0\). Then \(\pi_1\notin H_{n,m}\), so front insertion at \(s_0\) fails the adjacency condition.

Now consider a site \(s_j\) with \(1\le j<n\). If
\[
\pi_1,\ldots,\pi_j
\]
is not decreasing, then \(s_j\) is inactive by Lemma~\ref{lem:decreasing-prefix}. If this prefix is decreasing, then
\[
\pi_j\le \pi_1<n+1-m,
\]
so \(\pi_j\notin H_{n,m}\). Since \(\pi_j\) would be adjacent to the new maximum \(n+1\), the adjacency condition fails. Thus no such \(s_j\) is active.

It remains only to consider the terminal site \(s_n\). If the entire permutation is not decreasing, then \(s_n\) is inactive by Lemma~\ref{lem:decreasing-prefix}. If the entire permutation is decreasing, then \(\pi_n=1\), and hence
\[
|(n+1)-\pi_n|=n>m,
\]
because \(n\ge m+1\). Thus \(s_n\) also fails the adjacency condition. Therefore, if \(h(\pi)=0\), the permutation has no active sites.

Now suppose that \(h(\pi)\ge 1\). Then \(\pi_1\in H_{n,m}\), so front insertion at \(s_0\) satisfies the adjacency condition. Since nothing precedes the new maximum, it also cannot create a new \(123\) pattern. Hence \(s_0\) is active.

Consider an internal site \(s_j\) with \(1\le j<n\). By Lemma~\ref{lem:decreasing-prefix}, the pattern condition is satisfied if and only if
\[
j\le d(\pi).
\]
Assume this condition holds. Then
\[
\pi_1>\pi_2>\cdots>\pi_j.
\]
The adjacency condition at \(s_j\) requires both \(\pi_j\) and \(\pi_{j+1}\) to lie in \(H_{n,m}\). Under the decreasing-prefix condition, this is equivalent to
\[
j+1\le h(\pi).
\]
Indeed, if \(j+1\le h(\pi)\), then both neighboring entries are high. Conversely, if \(\pi_j\) and \(\pi_{j+1}\) are high, then \(\pi_1,\ldots,\pi_j\) are all high as well, since they are larger than or equal to \(\pi_j\). Hence the first \(j+1\) entries are high.

Therefore an internal site \(s_j\) is active exactly when
\[
j\le d(\pi)
\qquad\text{and}\qquad
j\le h(\pi)-1.
\]
Equivalently,
\[
j\le \min\{d(\pi),h(\pi)-1\}.
\]

Finally, the terminal site \(s_n\) is inactive by the same argument used above: either the entire permutation is not decreasing, in which case Lemma~\ref{lem:decreasing-prefix} applies, or it is decreasing, in which case \(\pi_n=1\) is too far from \(n+1\).

Thus, writing
\[
r=\min\{d(\pi),h(\pi)-1\},
\]
we obtain
\[
\operatorname{Act}(\pi)=\{s_0,s_1,\ldots,s_r\}.
\]
\end{proof}

There is a useful way to read Proposition~\ref{prop:active-sites}. As we move from left to right through the possible insertion sites, only two things can stop the active region. One possibility is that an ascent appears in the prefix. Once this happens, inserting the new maximum anywhere later would complete that increasing pair to a \(123\) pattern. The other possibility is that, while the prefix is still decreasing, we encounter an entry outside \(H_{n,m}\). Every later entry in that decreasing prefix is then even smaller and hence also non-high. If the prefix later ceases to decrease, the first obstruction takes over. This is why active sites cannot occur sporadically: they must form one initial block. Since \(H_{n,m}\) contains only \(m\) values, this immediately gives the boundedness that makes the recursive construction finite.

\begin{cor}[Bounded number of active sites]
\label{cor:bounded-active}
For every $\pi\in\mathcal C_{n,m}$ with $n\ge m+1$,
\[
|\operatorname{Act}(\pi)|\le m.
\]
\end{cor}

\begin{proof}
If there are no active sites, the claim is immediate. Otherwise, Proposition~\ref{prop:active-sites} gives
\[
|\operatorname{Act}(\pi)|
=
r+1
\le h(\pi).
\]
But $h(\pi)$ counts an initial sequence of distinct values drawn from the $m$-element set $H_{n,m}$, so
\[
h(\pi)\le m.
\]
\end{proof}

We have now reached the main structural conclusion needed for enumeration: every admissible permutation has a unique parent obtained by deleting its maximum, while every parent has at most $m$ legal ways to produce a child. Significantly, those legal insertions are controlled entirely by a short initial segment consisting of high entries and by whether that segment remains decreasing. Hence, the rest of the permutation does not need to be remembered in full.

What remains is to determine exactly what information from this initial segment must be retained in order to predict the active sites of the next generation. Since the relevant entries all lie near the current maximum, the natural quantities to record are not their absolute values but their \emph{deficits from the maximum}. This leads to the active-profile state description of the next section.

\section{Active sites and finite-state enumeration}
\label{sec:finite-state}

The previous section reduced the insertion problem to a bounded initial portion of the permutation. Every permutation in $\mathcal C_{n,m}$ has a unique parent obtained by deleting its maximum, and its legal insertion sites form an initial segment whose length is at most $m$. This already suggests that, for fixed $m$, the full permutation contains far more information than is needed to determine its future extensions. The next task is to identify exactly what must be remembered.

There are two requirements for a useful state description. First, it should contain enough information to determine the active sites after the next maximum is inserted. Second, its possible values should depend only on $m$, not on the permutation length $n$. Once such a description is found, the enumeration problem becomes finite: permutations can be grouped according to state, and insertion of successive maxima becomes a transition between a fixed collection of states.

Throughout this section, unless stated otherwise, we assume that $n \ge m+1$.

\subsection{Active profiles}
\label{subsec:active-profiles}

Suppose that
\[
\pi=\pi_1\pi_2\cdots\pi_n\in\mathcal C_{n,m}
\]
has $\ell$ active sites. By the structure established in the previous
section,
\[
\operatorname{Act}(\pi)
=
\{s_0,s_1,\ldots,s_{\ell-1}\}.
\]
As a result, only the initial entries
\[
\pi_1,\ldots,\pi_\ell
\]
can influence where the next maximum may be inserted.

One could attempt to use these entries themselves as a state. That would be awkward, however, because their values change with $n$. For instance, the entry $n-2$ represents the same position relative to the maximum regardless of whether the permutation has length $20$ or $200$. What matters is not its absolute value but the fact that it lies two below the maximum. This motivates measuring each relevant entry by its \emph{deficit} from the current maximum.

\begin{definition}[Active profile]
\label{def:active-profile}
Let $\pi\in\mathcal C_{n,m}$ have $\ell\ge1$ active sites. For
$1\le i\le\ell$, define
\[
d_i=n-\pi_i.
\]
The \emph{active profile} of $\pi$ is
\[
\lambda_m(\pi)
=
(d_1,\ldots,d_\ell).
\]

If $\pi$ has no active sites, we write
\[
\lambda_m(\pi)=\dagger,
\]
and call $\dagger$ the \emph{dead state}.
\end{definition}

The deficit coordinates have a direct interpretation. A deficit of $0$ corresponds to the current maximum, a deficit of $1$ to the value immediately below it, and so on. Since every entry participating in an active site must belong to the high set
\[
H_{n,m}
=
\{n+1-m,\ldots,n\},
\]
every recorded deficit lies in
\[
\{0,1,\ldots,m-1\}.
\]

The profile also remembers the order in which these high values appear. This is essential: pattern avoidance does not merely care which high values occur near the beginning, but whether those values form an ascent. 

There is one slight asymmetry in the profile that is worth making explicit. If there are $\ell$ active sites, then
\[
s_0,s_1,\ldots,s_{\ell-1}
\]
are active. In particular, the first $\ell-1$ entries lie strictly before the final active site $s_{\ell-1}$, so they must form a decreasing prefix:
\[
\pi_1>\pi_2>\cdots>\pi_{\ell-1}.
\]
In deficit coordinates this becomes
\[
d_1<d_2<\cdots<d_{\ell-1}.
\]
The final entry $\pi_\ell$, however, lies immediately to the right of $s_{\ell-1}$. It must still be high enough to sit next to a new maximum, but it does not need to continue the decreasing run. Thus, $d_\ell$ is not initially forced to be larger than $d_{\ell-1}$.

\begin{definition}[Candidate profile]
\label{def:candidate-profile}
Fix $m\ge1$. A \emph{candidate profile} of length $\ell$, where
$1\le\ell\le m$, is a tuple
\[
\omega=(d_1,\ldots,d_\ell)
\]
such that
\[
d_i\in\{0,1,\ldots,m-1\},
\]
the $d_i$ are distinct, and
\[
d_1<d_2<\cdots<d_{\ell-1}.
\]

A candidate profile is called \emph{realizable} if
\[
\omega=\lambda_m(\pi)
\]
for some $\pi\in\mathcal C_{n,m}$ and some $n$.
\end{definition}

At this stage we have compressed the relevant part of a permutation from its actual values to a tuple of at most $m$ integers chosen from
$\{0,\ldots,m-1\}$. The next question is whether this tuple really contains enough information. In other words, if two permutations have the same active profile, do they necessarily behave in the same way when the next maximum is inserted? To answer that, we first determine how deficits change under insertion.

\subsection{Profile evolution under maximum insertion}
\label{subsec:profile-evolution}

Suppose the current permutation has length $n$, and we insert the new maximum $n+1$. Two simple things happen to the deficit coordinates.

First, the newly inserted value $n+1$ has deficit
\[
(n+1)-(n+1)=0.
\]
Second, if an old entry had deficit
\[
d=n-\pi_i,
\]
then its deficit from the new maximum is
\[
(n+1)-\pi_i=d+1.
\]

Thus insertion has an especially simple effect in deficit coordinates:
\[
\text{old deficits increase by }1,
\qquad
\text{the new maximum contributes }0.
\]

There is only one complication. An old deficit equal to $m-1$ becomes $m$. Such an entry is no longer high enough to be adjacent to the \emph{next} maximum. Since active sites always form an initial segment, once the first such obstruction is reached, no later entry can remain relevant. We encode this by a trimming operation.

\begin{definition}[Trimming map]
\label{def:trim}
Let
\[
x=(x_1,\ldots,x_r)
\]
be a finite ordered list with entries in
\[
\{0,1,\ldots,m\}.
\]
If $m$ occurs in the list, let
\[
q=\min\{i:x_i=m\}
\]
be the position of its first occurrence. If $m$ does not occur, set $q=r+1$.

Define
\[
\operatorname{trim}_m(x_1,\ldots,x_r)
=
(x_1,\ldots,x_{q-1}).
\]
If $q=1$, so that the resulting list is empty, we interpret the output as the dead state $\dagger$.
\end{definition}

The meaning of $\operatorname{trim}_m$ is simple. As the maximum increases, old high entries become one unit farther from it. The first entry whose deficit reaches $m$ has become too small to support another maximum insertion. Because active sites form an initial segment, that entry and everything after it may be discarded from the state. We can now write the child profile produced by each active insertion site.

\begin{definition}[Profile transition maps]
\label{def:phi}
Let
\[
\omega=(d_1,\ldots,d_\ell)
\]
be a nondead candidate profile.

For front insertion at $s_0$, define
\[
\Phi_{m,0}(\omega)
=
\operatorname{trim}_m
\bigl(
0,d_1+1,\ldots,d_\ell+1
\bigr).
\]

For $1\le j\le\ell-1$, define
\[
\Phi_{m,j}(\omega)
=
\operatorname{trim}_m
\bigl(
d_1+1,\ldots,d_j+1,0
\bigr).
\]
\end{definition}

The front-insertion formula follows directly from the deficit update: $n+1$ becomes the first entry and contributes $0$, while every old deficit increases by $1$. The internal-insertion formula deserves one additional comment, however. If $n+1$ is inserted after $\pi_j$, then
\[
\pi_j<n+1
\]
is an ascent. Any insertion site strictly after the newly inserted maximum would thus have a prefix containing this ascent. Such a site cannot be active, because inserting $n+2$ there would complete the ascent to a $123$ pattern. Hence nothing after the new maximum needs to be retained in the child's profile.

The following proposition is the point at which the state description becomes useful.

\begin{prop}[Profile sufficiency]
\label{prop:profile-sufficiency}
Let $\pi\in\mathcal C_{n,m}$ have nondead active profile
\[
\lambda_m(\pi)=\omega=(d_1,\ldots,d_\ell).
\]
For every active-site index
\[
0\le j\le\ell-1,
\]
we have
\[
\lambda_m(I_j(\pi))
=
\Phi_{m,j}(\omega).
\]

Consequently, if two permutations have the same active profile, then insertion at correspondingly indexed active sites produces children with the same active profiles.
\end{prop}

\begin{proof}
Let
\[
\tau=I_j(\pi)
\]
be the permutation obtained by inserting $n+1$.

We must determine the active sites available for inserting the next maximum $n+2$. An entry of $\tau$ can be adjacent to $n+2$ precisely when its deficit from the current maximum $n+1$ is at most $m-1$. Thus a transformed deficit equal to $m$ is exactly the first possible obstruction coming from the adjacency condition.

Suppose first that $j=0$. Then
\[
\tau=(n+1),\pi_1,\pi_2,\ldots,
\]
so the relevant transformed deficits begin
\[
0,d_1+1,d_2+1,\ldots,d_\ell+1.
\]
Until the first occurrence of $m$, all recorded entries remain high enough to sit beside $n+2$. Moreover,
\[
d_1<\cdots<d_{\ell-1}
\]
implies
\[
0<d_1+1<\cdots<d_{\ell-1}+1,
\]
so the corresponding values form a decreasing initial segment.

Hence, all sites represented before the first transformed deficit $m$ remain active. At the first deficit $m$, the adjacency condition fails, and since active sites form an initial segment, all later sites are inactive. If no transformed deficit equals $m$, then the old obstruction that ended the parent's active region remains an obstruction after front insertion: an old ascent is not removed, and an old entry that was too small relative to $n$ remains too small relative to $n+1$.

Hence
\[
\lambda_m(\tau)
=
\operatorname{trim}_m
\bigl(
0,d_1+1,\ldots,d_\ell+1
\bigr)
=
\Phi_{m,0}(\omega).
\]

Now suppose $1\le j\le\ell-1$. The relevant transformed deficits are
\[
d_1+1,\ldots,d_j+1,0.
\]
As before, the sites remain active until the first transformed deficit equal to $m$. If such a deficit occurs, the adjacency condition fails there and all later sites are inactive.

If no transformed deficit equals $m$, then the site immediately before the new maximum remains active. No site strictly after $n+1$ can be active, however, because every such prefix contains the ascent
\[
\pi_j<n+1.
\]
Thus
\[
\lambda_m(\tau)
=
\operatorname{trim}_m
\bigl(
d_1+1,\ldots,d_j+1,0
\bigr)
=
\Phi_{m,j}(\omega).
\]

In either case, the child's profile is determined entirely by the parent's profile and the chosen active-site index.
\end{proof}

Proposition~\ref{prop:profile-sufficiency} is the main reason for introducing active profiles. The full permutation is no longer needed to predict future extensions. Once $\lambda_m(\pi)$ is known, every possible child state is known. In this sense, the active profile contains all of the information relevant to future maximum insertions. There is still one issue. Definition~\ref{def:candidate-profile} described all tuples having the obvious local shape of an active profile, but some of those tuples may never occur in an actual permutation. Before constructing the finite-state system, we determine exactly which candidate profiles are realizable.

\subsection{Realizable profiles}
\label{subsec:realizable-profiles}

The only possible failure of monotonicity in a candidate profile occurs at its final coordinate. Thus a candidate profile has one of two shapes: either
\[
d_1<d_2<\cdots<d_\ell,
\]
or
\[
d_1<\cdots<d_{\ell-1}
\qquad\text{and}\qquad
d_\ell<d_{\ell-1}.
\]

The first case corresponds to a fully decreasing recorded segment in the original permutation. The second corresponds to a recorded segment ending in an ascent.

The second case is much more restrictive than it first appears. Suppose
\[
d_\ell<d_{\ell-1}.
\]
Then in terms of values,
\[
\pi_{\ell-1}
=
n-d_{\ell-1}
<
n-d_\ell
=
\pi_\ell.
\]
We already have an ascent. To avoid creating a $123$ pattern, every value larger than $\pi_\ell$ must have appeared earlier. In deficit language, this means every deficit smaller than $d_\ell$ must already have occurred. The natural way to express this condition is through the smallest missing deficit.

\begin{definition}[Minimum excluded value]
\label{def:mex}
For a finite set
\[
S\subseteq\mathbb Z_{\ge0},
\]
define
\[
\operatorname{mex}(S)
=
\min\bigl(\mathbb Z_{\ge0}\setminus S\bigr).
\]
\end{definition}

\begin{theorem}[Characterization of realizable profiles]
\label{thm:realizable-profiles}
Let
\[
\omega=(d_1,\ldots,d_\ell)
\]
be a candidate profile, where $1\le\ell\le m$.

Then $\omega$ is realizable if and only if one of the following holds:

\begin{enumerate}
    \item
    \[
    d_1<d_2<\cdots<d_\ell;
    \]

    \item $\ell\ge2$,
    \[
    d_\ell<d_{\ell-1},
    \]
    and
    \[
    d_\ell
    =
    \operatorname{mex}
    \{d_1,\ldots,d_{\ell-1}\}.
    \]
\end{enumerate}

Moreover, every realizable profile already occurs in
$\mathcal C_{m+1,m}$.
\end{theorem}

\begin{proof}
Suppose first that
\[
\omega=\lambda_m(\pi)
=
(d_1,\ldots,d_\ell)
\]
for some $\pi\in\mathcal C_{n,m}$.

We already know
\[
d_1<\cdots<d_{\ell-1}.
\]
If
\[
d_{\ell-1}<d_\ell,
\]
then the entire profile is increasing, giving the first case.

It remains to consider
\[
r:=d_\ell<d_{\ell-1}.
\]
The last two recorded values are
\[
\pi_{\ell-1}
=
n-d_{\ell-1}
\qquad\text{and}\qquad
\pi_\ell
=
n-r,
\]
so
\[
\pi_{\ell-1}<\pi_\ell.
\]

We claim that every deficit smaller than $r$ must already occur among
\[
d_1,\ldots,d_{\ell-1}.
\]
Suppose not. Then there is some
\[
q<r
\]
with
\[
q\notin\{d_1,\ldots,d_{\ell-1}\}.
\]
The value
\[
n-q
\]
has thus not appeared among the first $\ell$ entries and must occur
later in $\pi$.

Since
\[
q<r<d_{\ell-1},
\]
we have
\[
n-d_{\ell-1}
<
n-r
<
n-q.
\]
These values occur in increasing positional order as
\[
\pi_{\ell-1},
\pi_\ell,
n-q,
\]
and, as a result, form a $123$ pattern, a contradiction.

Hence
\[
\{0,1,\ldots,r-1\}
\subseteq
\{d_1,\ldots,d_{\ell-1}\}.
\]
Because all deficits in the profile are distinct,
\[
r\notin\{d_1,\ldots,d_{\ell-1}\}.
\]
Therefore
\[
r
=
\operatorname{mex}
\{d_1,\ldots,d_{\ell-1}\}.
\]
This proves necessity.

We now prove sufficiency, and at the same time show that every realizable
profile occurs already at length
\[
n=m+1.
\]
At this length, the adjacency condition is automatic because the largest possible difference between two values is
\[
n-1=m.
\]
Thus we only need to construct a $123$-avoiding permutation with the desired active profile.

First suppose
\[
d_1<d_2<\cdots<d_\ell.
\]
Begin with
\[
n-d_1,\,
n-d_2,\,
\ldots,\,
n-d_\ell,
\]
then place $1$, and finally place all remaining unused values in decreasing order. Thus
\[
\pi
=
(n-d_1)(n-d_2)\cdots(n-d_\ell)\,1\,R^\downarrow,
\]
where $R^\downarrow$ denotes the unused values written in decreasing order.

The first $\ell$ entries are high and decreasing, so
\[
s_0,\ldots,s_{\ell-1}
\]
are active. The following site is inactive because the next entry is $1$, which is not high.

The permutation avoids $123$: the initial recorded segment is decreasing, the suffix $R^\downarrow$ is decreasing, and any increasing pair beginning with $1$ cannot be extended to a third larger entry because the entries following $1$ occur in decreasing order.

Hence
\[
\lambda_m(\pi)
=
(d_1,\ldots,d_\ell).
\]

Now suppose
\[
r:=d_\ell
=
\operatorname{mex}\{d_1,\ldots,d_{\ell-1}\}
<
d_{\ell-1}.
\]
Begin with
\[
n-d_1,\ldots,n-d_{\ell-1},n-r
\]
and place all unused values afterward in decreasing order:
\[
\pi
=
(n-d_1)\cdots(n-d_{\ell-1})(n-r)R^\downarrow.
\]

The first $\ell-1$ entries decrease, while
\[
n-d_{\ell-1}<n-r,
\]
so the recorded segment ends in an ascent.

Because $r$ is the mex of the preceding deficits, every deficit smaller than $r$ has already appeared. Equivalently, every value larger than $n-r$ has already occurred before the suffix. Thus every entry of $R^\downarrow$ is smaller than $n-r$. The only ascent in the recorded segment ends at $n-r$, and there is no later entry larger than $n-r$ with which that ascent could form a $123$ pattern. The suffix itself is decreasing. Hence $\pi$ avoids $123$.

Again,
\[
s_0,\ldots,s_{\ell-1}
\]
are active, while the next site is inactive because its prefix contains the ascent
\[
n-d_{\ell-1}<n-r.
\]
Thus
\[
\lambda_m(\pi)
=
(d_1,\ldots,d_{\ell-1},r).
\]

Finally, the dead state is realized at length $m+1$ by
\[
1,n,n-1,\ldots,2.
\]
Its first site is inactive because its first entry is not high, and the permutation avoids $123$ because everything after $1$ is decreasing.
\end{proof}

The mex condition has a simple interpretation that is easy to lose in the notation. If the recorded segment ends in an ascent
\[
\pi_{\ell-1}<\pi_\ell,
\]
then $\pi_\ell$ must already be larger than every value that remains unseen. Otherwise some larger later value would extend this ascent to a $123$ pattern. Saying that
\[
d_\ell
=
\operatorname{mex}\{d_1,\ldots,d_{\ell-1}\}
\]
is exactly the deficit-coordinate version of that statement.

We may now define the actual state space.

\begin{definition}[Realizable state space]
\label{def:omega}
For fixed $m$, let
\[
\Omega_m
\]
denote the set consisting of the dead state $\dagger$ together with all realizable active profiles characterized in Theorem~\ref{thm:realizable-profiles}.
\end{definition}

Unlike the original permutations, whose length grows without bound, $\Omega_m$ depends only on $m$. We can in fact count its states exactly.

\begin{cor}[Profiles of fixed length]
\label{cor:profiles-fixed-length}
Let $N_{m,\ell}$ be the number of realizable nondead profiles of length $\ell$. Then
\[
N_{m,\ell}
=
\binom{m+1}{\ell}-1.
\]
\end{cor}

\begin{proof}
We count the two types from Theorem~\ref{thm:realizable-profiles} separately.

A fully increasing profile
\[
d_1<d_2<\cdots<d_\ell
\]
is determined by choosing an \(\ell\)-element subset of
\[
\{0,1,\ldots,m-1\}.
\]
Hence there are
\[
\binom{m}{\ell}
\]
profiles of this type.

If \(\ell=1\), every realizable profile is of this form, so
\[
N_{m,1}=m=\binom{m+1}{1}-1,
\]
as required. As a result, we may assume that \(\ell\ge 2\).

Now consider terminal-ascent profiles. Choose the set
\[
S=\{d_1,\ldots,d_{\ell-1}\}.
\]
There are
\[
\binom{m}{\ell-1}
\]
possible \((\ell-1)\)-element subsets of
\(\{0,\ldots,m-1\}\). Once \(S\) is chosen, the final coordinate is forced to be
\[
d_\ell=\operatorname{mex}(S).
\]

Exactly one choice of \(S\) does not produce a terminal ascent, namely
\[
S=\{0,1,\ldots,\ell-2\}.
\]
For this set,
\[
\operatorname{mex}(S)=\ell-1>\max S,
\]
so appending the mex produces a fully increasing profile. For every other choice of \(S\),
\[
\operatorname{mex}(S)<\max S,
\]
and hence appending the mex produces exactly one terminal-ascent profile.

Thus there are
\[
\binom{m}{\ell-1}-1
\]
terminal-ascent profiles. Adding the two types and applying Pascal's identity gives
\[
N_{m,\ell}
=
\binom{m}{\ell}
+
\binom{m}{\ell-1}
-1
=
\binom{m+1}{\ell}-1.
\]
\end{proof}

\begin{cor}[Exact number of states]
\label{cor:state-count}
The realizable state space has cardinality
\[
|\Omega_m|
=
2^{m+1}-m-1.
\]
\end{cor}

\begin{proof}
Including the dead state,
\[
|\Omega_m|
=
1+
\sum_{\ell=1}^{m}
\left(
\binom{m+1}{\ell}-1
\right).
\]
Therefore
\[
|\Omega_m|
=
1+
\left(
\sum_{\ell=1}^{m}
\binom{m+1}{\ell}
\right)
-m.
\]
Since
\[
\sum_{\ell=1}^{m}
\binom{m+1}{\ell}
=
2^{m+1}-2,
\]
we obtain
\[
|\Omega_m|
=
2^{m+1}-m-1.
\]
\end{proof}

Thus the number of states may grow quickly with $m$, but for every fixed $m$ it is finite and, crucially, independent of $n$. This is exactly what is needed for a transfer-matrix enumeration.

\begin{remark}
Corollary~\ref{cor:state-count} gives the exact number of realizable active profiles. It does not claim that this is the smallest possible state space for enumeration. Two distinct profiles could, in principle, have identical future counting behavior and, consequently, admit a further quotient. The state space $\Omega_m$ is exact for the active profile description developed here.
\end{remark}

Before defining the matrix, we record one useful consequence of realizability.

\begin{cor}[Closure under profile transitions]
\label{cor:profile-closure}
If
\[
\omega\in\Omega_m
\]
is nondead and
\[
0\le j\le|\omega|-1,
\]
then
\[
\Phi_{m,j}(\omega)\in\Omega_m.
\]
\end{cor}

\begin{proof}
By Theorem~\ref{thm:realizable-profiles}, choose a permutation
\[
\pi\in\mathcal C_{m+1,m}
\]
with
\[
\lambda_m(\pi)=\omega.
\]
Insertion at the active site $s_j$ produces a valid child. By Proposition~\ref{prop:profile-sufficiency}, the child's active profile is
\[
\Phi_{m,j}(\omega).
\]
Hence this profile is realizable and belongs to $\Omega_m$.
\end{proof}

\subsection{The transfer system}
\label{subsec:transfer-system}

We now have all the ingredients of a finite-state process. The vertices are the states in $\Omega_m$, and each active insertion site gives a transition from a parent state to a child state.

\begin{definition}[Transfer matrix]
\label{def:transfer-matrix}
For fixed \(m\), define the matrix \(T_m\), with rows and columns indexed by \(\Omega_m\), as follows. For every nondead state \(\omega\in\Omega_m\), set
\[
(T_m)_{\omega,\eta}
=
\#\left\{
j\in\{0,\ldots,|\omega|-1\}:
\Phi_{m,j}(\omega)=\eta
\right\}.
\]
For the dead state, set
\[
(T_m)_{\dagger,\eta}=0
\]
for every \(\eta\in\Omega_m\).
\end{definition}

Thus $(T_m)_{\omega,\eta}$ counts the number of active insertion sites of a permutation in state $\omega$ that produce a child in state $\eta$. Conceptually, two different active sites could conceivably lead to the same child state, in which case some matrix entries would exceed $1$. In fact, this never occurs.

\begin{prop}[$0$--$1$ structure and row sums]
\label{prop:zero-one}
Let
\[
\omega=(d_1,\ldots,d_\ell)\in\Omega_m
\]
be nondead. Then
\[
\Phi_{m,0}(\omega),
\Phi_{m,1}(\omega),
\ldots,
\Phi_{m,\ell-1}(\omega)
\]
are pairwise distinct.

Consequently,
\[
(T_m)_{\omega,\eta}\in\{0,1\}
\]
for all $\omega,\eta\in\Omega_m$, and
\[
\sum_{\eta\in\Omega_m}
(T_m)_{\omega,\eta}
=
|\omega|.
\]
The row sum of the dead state is $0$.
\end{prop}

\begin{proof}
If \(\ell=1\), there is only one active-site index, namely \(j=0\). Accordingly, the pairwise-distinctness assertion is immediate, and the row contains exactly one nonzero entry. Thus, we may assume throughout the remainder of the proof that \(\ell\ge 2\).

Front insertion gives
\[
\Phi_{m,0}(\omega)
=
\operatorname{trim}_m
(0,d_1+1,\ldots,d_\ell+1).
\]
Its first coordinate is always $0$, since that initial $0$ can never be removed by trimming.

Now let
\[
1\le j\le\ell-2.
\]
Because
\[
d_1<\cdots<d_{\ell-1},
\]
we have
\[
d_j<d_{\ell-1}\le m-1,
\]
so
\[
d_j\le m-2.
\]
Therefore none of
\[
d_1+1,\ldots,d_j+1
\]
equals $m$, and no trimming occurs. Hence
\[
\Phi_{m,j}(\omega)
=
(d_1+1,\ldots,d_j+1,0).
\]
This profile begins with a positive deficit, ends in $0$, and has length $j+1$.

Two such profiles corresponding to distinct $j$ have different lengths and thus cannot agree. None can agree with the front-insertion child, which begins with $0$. It remains to consider $j=\ell-1$.

If
\[
d_{\ell-1}\le m-2,
\]
then again no trimming occurs, and
\[
\Phi_{m,\ell-1}(\omega)
=
(d_1+1,\ldots,d_{\ell-1}+1,0).
\]
This has length $\ell$, whereas all earlier internal children have smaller length, and it begins with a positive deficit, so it differs from the front-insertion child.

If instead
\[
d_{\ell-1}=m-1,
\]
then
\[
\Phi_{m,\ell-1}(\omega)
=
\operatorname{trim}_m
(d_1+1,\ldots,d_{\ell-2}+1,m,0)
\]
and hence
\[
\Phi_{m,\ell-1}(\omega)
=
(d_1+1,\ldots,d_{\ell-2}+1),
\]
with the empty list interpreted as $\dagger$.

If nonempty, this profile contains no $0$, so it differs both from the front-insertion child, which begins with $0$, and from every untrimmed internal child, which ends with $0$. If it is $\dagger$, the distinction is immediate.

Thus all $\ell$ active sites produce distinct child states. The claims about the matrix entries and row sums follow.
\end{proof}

The $0$--$1$ structure makes the transfer matrix especially concrete: $T_m$ is simply the adjacency matrix of the directed graph of possible state transitions. A state of length $\ell$ has exactly $\ell$ outgoing edges, one for each active site. We now use this matrix to count permutations.

\begin{definition}[State-count vector]
\label{def:state-count}
For $n\ge m+1$ and $\omega\in\Omega_m$, define
\[
u_{n,m}(\omega)
=
\#\left\{
\pi\in\mathcal C_{n,m}:
\lambda_m(\pi)=\omega
\right\}.
\]
Let
\[
u_{n,m}
=
\bigl(u_{n,m}(\omega)\bigr)_{\omega\in\Omega_m}
\]
be the resulting row vector.
\end{definition}

The transfer matrix now has a direct counting meaning. Every permutation at level $n+1$ has a unique parent at level $n$, obtained by deleting its maximum. Conversely, every child is obtained by inserting the new maximum into one active site of its parent. This is exactly what matrix multiplication records. In other words, when the relevant information associated with an insertion process can be represented by finitely many labels, the resulting generating tree can be encoded by a finite transition system, leading to rational generating functions~\cite{Vatter2006,AlbertLintonRuskuc2005}.

\begin{theorem}[Finite-state transfer relation]
\label{thm:transfer-relation}
For every fixed $m\ge1$ and every $n\ge m+1$,
\[
u_{n+1,m}
=
u_{n,m}T_m.
\]
Consequently,
\[
u_{n,m}
=
u_{m+1,m}T_m^{\,n-m-1}
\]
for every $n\ge m+1$.
\end{theorem}

\begin{proof}
Fix a state
\[
\eta\in\Omega_m.
\]
We count the permutations in $\mathcal C_{n+1,m}$ having profile $\eta$.

Every
\[
\tau\in\mathcal C_{n+1,m}
\]
has a unique parent obtained by deleting $n+1$. By closure under maximum deletion, this parent lies in $\mathcal C_{n,m}$.

Now fix a parent state $\omega$. There are
\[
u_{n,m}(\omega)
\]
parents in that state. For each such parent, $(T_m)_{\omega,\eta}$ counts exactly the active sites whose insertion gives a child in state $\eta$. Hence
\[
u_{n+1,m}(\eta)
=
\sum_{\omega\in\Omega_m}
u_{n,m}(\omega)
(T_m)_{\omega,\eta}.
\]
This is precisely the $\eta$-coordinate of
\[
u_{n,m}T_m.
\]
Since the identity holds for every $\eta$,
\[
u_{n+1,m}
=
u_{n,m}T_m.
\]

Iterating from level $m+1$ gives
\[
u_{n,m}
=
u_{m+1,m}T_m^{\,n-m-1}.
\]
\end{proof}

It is worth emphasizing exactly what has occurred. The original problem asks us to count permutations of arbitrarily large size, and Theorem~\ref{thm:transfer-relation} replaces this with repeated multiplication by one fixed matrix whose dimension is
\[
2^{m+1}-m-1.
\]
For fixed $m$, the dependence on $n$ has been reduced entirely to the number of times the same finite transition rule is applied. To make this description completely explicit, only one piece remains --- the starting vector at level $m+1$.

\subsection{The initial state vector}
\label{subsec:initial-vector}

The level
\[
n=m+1
\]
is the natural place to initialize the transfer process. At this length the adjacency condition is vacuous, because the largest possible difference between two values is
\[
m.
\]
Therefore
\[
\mathcal C_{m+1,m}
=
\operatorname{Av}_{m+1}(123).
\]

\begin{definition}[Initial state vector]
\label{def:initial-vector}
Define
\[
v_m:=u_{m+1,m}.
\]
Thus
\[
v_m(\omega)
=
\#\left\{
\pi\in\mathcal C_{m+1,m}:
\lambda_m(\pi)=\omega
\right\}.
\]
\end{definition}

For the dead state and for fully increasing deficit profiles, the initial multiplicity is $1$. Terminal-ascent profiles are the only states that admit more than one completion. Their multiplicities reduce to a familiar Catalan-triangle count~\cite{DeSantisEtAl2013}.

\begin{definition}[Constrained Catalan count]
\label{def:B}
For integers $0\le h\le N$, let $B(N,h)$ denote the number of $123$-avoiding permutations of length $N$ in which the $h$ largest values occur from left to right in decreasing order.
\end{definition}

\begin{lemma}[Catalan-triangle formula]
\label{lem:catalan-triangle}
For $0\le h\le N$,
\[
B(N,h)
=
\frac{h+1}{N+1}
\binom{2N-h}{N}.
\]
\end{lemma}

\begin{proof}
It is convenient to convert the condition on the \(h\) largest values into a condition on an initial decreasing run. Inversion preserves \(123\)-avoidance, and reverse-complement also preserves \(123\)-avoidance. Their composition gives a bijection under which the condition that the \(h\) largest values occur in decreasing order becomes the condition that the first \(h\) entries form a decreasing sequence.

Let \(E(N,r)\) denote the number of \(123\)-avoiding permutations of length \(N\) whose initial decreasing run has length exactly \(r\), for \(1\le r\le N\). For convenience, set
\[
E(N,0)=0.
\]

Consider a \(123\)-avoiding permutation of length \(N\) whose initial decreasing run has length \(s\). The new maximum may be inserted at any of the sites
\[
s_0,s_1,\ldots,s_s.
\]
Front insertion produces an initial decreasing run of length \(s+1\), while insertion after exactly \(j\) entries, with \(1\le j\le s\), produces an initial decreasing run of length \(j\). Therefore, for \(1\le r\le N+1\),
\[
E(N+1,r)=\sum_{s=r-1}^{N}E(N,s).
\]

We claim that
\[
E(N,r)
=
\binom{2N-r-1}{N-1}
-
\binom{2N-r-1}{N}
\]
for \(0\le r\le N\), where the formula at \(r=0\) indeed gives \(0\).

For \(N=1\), we have \(E(1,0)=0\) and \(E(1,1)=1\), so the formula
holds. Suppose it holds at length \(N\). Then
\begin{align*}
E(N+1,r)
&=
\sum_{s=r-1}^{N}
\left[
\binom{2N-s-1}{N-1}
-
\binom{2N-s-1}{N}
\right].
\end{align*}
Setting \(t=2N-s-1\), the first sum becomes
\[
\sum_{t=N-1}^{2N-r}\binom{t}{N-1}
=
\binom{2N-r+1}{N},
\]
by the hockey-stick identity, while the second becomes
\[
\sum_{t=N-1}^{2N-r}\binom{t}{N}
=
\binom{2N-r+1}{N+1}.
\]
Hence
\[
E(N+1,r)
=
\binom{2N-r+1}{N}
-
\binom{2N-r+1}{N+1},
\]
which is precisely
\[
E(N+1,r)
=
\binom{2(N+1)-r-1}{(N+1)-1}
-
\binom{2(N+1)-r-1}{N+1}.
\]
This completes the induction.

For \(h\ge 1\), requiring the first \(h\) entries to decrease is equivalent to requiring the initial decreasing run to have length at least \(h\). Therefore
\begin{align*}
B(N,h)
&=
\sum_{r=h}^{N}E(N,r) \\
&=
\binom{2N-h}{N}
-
\binom{2N-h}{N+1}.
\end{align*}
Factoring the right-hand side gives
\begin{align*}
B(N,h)
&=
\left(
1-\frac{N-h}{N+1}
\right)
\binom{2N-h}{N} \\
&=
\frac{h+1}{N+1}\binom{2N-h}{N}.
\end{align*}

For \(h=0\), the condition is vacuous, so
\[
B(N,0)=C_N
=
\frac{1}{N+1}\binom{2N}{N},
\]
which is the same formula with \(h=0\).
\end{proof}

We can now give every coordinate of $v_m$ explicitly.

\begin{theorem}[Explicit initial vector]
\label{thm:initial-vector}
Let
\[
v_m=u_{m+1,m}.
\]

For the dead state,
\[
v_m(\dagger)=1.
\]

If
\[
\omega=(d_1,\ldots,d_\ell)
\]
is a realizable profile satisfying
\[
d_1<d_2<\cdots<d_\ell,
\]
then
\[
v_m(\omega)=1.
\]

Finally, suppose
\[
\omega=(d_1,\ldots,d_{\ell-1},r)
\]
is a terminal-ascent profile. Put
\[
s=d_{\ell-1},
\qquad
N=m+1-\ell,
\qquad
h=s-\ell+1.
\]
Then
\[
v_m(\omega)
=
B(N,h)
=
\frac{h+1}{N+1}
\binom{2N-h}{N}.
\]
\end{theorem}

\begin{proof}
Throughout the proof,
\[
n=m+1.
\]
At this level the high values are
\[
2,3,\ldots,m+1,
\]
so $1$ is the unique value that is not high.

First consider the dead state. For $s_0$ to be inactive, the first entry must be $1$. Once $1$ occurs first, all remaining entries must appear in decreasing order; otherwise an ascent among the later entries, together with the initial $1$, would form a $123$ pattern. Thus the unique dead permutation is
\[
1,m+1,m,\ldots,2,
\]
and
\[
v_m(\dagger)=1.
\]

Now suppose
\[
\omega=(d_1,\ldots,d_\ell)
\]
is fully increasing. The profile fixes the first $\ell$ entries:
\[
n-d_1,\ldots,n-d_\ell,
\]
and these values form a decreasing sequence.

Because the profile has exactly $\ell$ active sites, the following site must be inactive. The prefix still contains no ascent, so the failure cannot come from the $123$ condition. It must come from adjacency. At length $m+1$, the only non-high value is $1$, so the next entry is forced to be $1$.

Everything after this $1$ must then occur in decreasing order, since any later ascent together with $1$ would form a $123$ pattern. Therefore there is exactly one completion, and
\[
v_m(\omega)=1.
\]

Finally, suppose
\[
\omega=(d_1,\ldots,d_{\ell-1},r)
\]
is a terminal-ascent profile, and let
\[
s=d_{\ell-1}.
\]
The final two recorded values are
\[
a=n-s
\qquad\text{and}\qquad
b=n-r.
\]
Since
\[
r<s,
\]
we have
\[
a<b.
\]

Because
\[
r=
\operatorname{mex}\{d_1,\ldots,d_{\ell-1}\},
\]
every deficit smaller than $r$ has already occurred. Equivalently, every value larger than $b$ has already appeared in the recorded prefix. Hence every value remaining for the suffix is smaller than $b$.

Let $\sigma$ denote the suffix consisting of the remaining
\[
N=n-\ell=m+1-\ell
\]
values.

We claim that the full permutation avoids $123$ if and only if:

\begin{enumerate}
    \item $\sigma$ avoids $123$;
    \item the remaining values larger than $a$ occur in decreasing order
    within $\sigma$.
\end{enumerate}

The first condition is obviously necessary. For the second, suppose two suffix values
\[
a<y<z
\]
occur in the order $y,z$. Then
\[
a,y,z
\]
forms a $123$ pattern, so all suffix values larger than $a$ must occur decreasingly.

Conversely, assume both conditions. A $123$ pattern cannot lie entirely in the suffix by the first condition. It cannot use two entries from the recorded prefix followed by one suffix entry: every ascent in the recorded prefix ends at $b$, while every suffix value is smaller than $b$.

Finally, suppose a $123$ pattern uses one recorded entry followed by two suffix entries. Among the recorded entries,
\[
a=n-s
\]
is the smallest. Thus any two suffix entries capable of completing such a pattern must both exceed $a$ and occur in increasing order. This is excluded by the second condition.

It remains to count how many remaining suffix values exceed $a$.

The values larger than $a=n-s$ correspond to deficits
\[
0,1,\ldots,s-1.
\]
Among these $s$ deficits, exactly $\ell-1$ have already appeared in the recorded profile: the $\ell-2$ deficits
\[
d_1,\ldots,d_{\ell-2}
\]
together with the final deficit $r$. Hence
\[
h
=
s-(\ell-1)
=
s-\ell+1
\]
such values remain.

These are precisely the $h$ largest values of the suffix. After standardizing the suffix, we thus need a $123$-avoiding permutation of length $N$ in which the $h$ largest values occur in decreasing order. By Lemma~\ref{lem:catalan-triangle},
\[
v_m(\omega)
=
B(N,h)
=
\frac{h+1}{N+1}
\binom{2N-h}{N}.
\]
\end{proof}

\begin{cor}
\label{cor:initial-positive}
Every coordinate of $v_m$ is positive, and
\[
v_m\mathbf 1=C_{m+1}.
\]
\end{cor}

\begin{proof}
Positivity follows from Theorem~\ref{thm:initial-vector}. Since the adjacency condition is vacuous at length $m+1$, the states partition all $123$-avoiding permutations of length $m+1$. Their total number is, accordingly, the Catalan number
\[
C_{m+1}.
\]
\end{proof}

We may now combine the transition rule with the explicit initial vector.

\begin{theorem}[Finite-state enumeration]
\label{thm:finite-state-enumeration}
For every fixed $m\ge1$ and every $n\ge m+1$,
\[
u_{n,m}
=
v_m T_m^{\,n-m-1}.
\]
If
\[
a_{n,m}=|\mathcal C_{n,m}|
\]
and $\mathbf 1$ denotes the column vector of ones indexed by $\Omega_m$, then
\[
a_{n,m}
=
v_m T_m^{\,n-m-1}\mathbf 1.
\]
\end{theorem}

Theorem~\ref{thm:finite-state-enumeration} completes the finite-state reduction. For fixed $m$, the original permutation problem is now described entirely by three finite objects:

\[
\Omega_m,
\qquad
T_m,
\qquad
v_m.
\]

The state space $\Omega_m$ records exactly the realizable extension types, the matrix $T_m$ records how those types change when a new maximum is inserted, and the vector $v_m$ records how many permutations occupy each state when the stable insertion process begins at length $m+1$.

\begin{cor}[Rational generating function]
\label{cor:rational-gen}
For every fixed $m\ge1$, let 
\[
A_m(z) = \sum_{n\ge0}a_{n,m}z^n.
\]
With the convention $a_{0,m} = 1$, we have 
\[
A_m(z) = \sum_{n=0}^{m}C_nz^n+z^{m+1}v_m(I-zT_m)^{-1}\mathbf{1}.
\]
Consequently, $A_m(z)\in \mathbb{Q}(z)$, and the sequence $(a_{n,m})_{n\ge0}$ eventually satisfies a homogeneous linear recurrence with integer coefficients.
\end{cor}
\begin{proof}
For \(0\le n\le m+1\), the adjacency condition is vacuous, since the largest possible difference between two entries of a permutation of \([n]\) is \(n-1\le m\). Hence
\[
a_{n,m}=C_n
\qquad
(0\le n\le m+1).
\]

For \(k\ge 0\), Theorem~5 gives
\[
a_{m+1+k,m}=v_mT_m^k\mathbf{1}.
\]
Therefore
\begin{align*}
A_m(z)
&=
\sum_{n=0}^{m} C_nz^n
+
\sum_{k\ge 0}
v_mT_m^k\mathbf{1}\,z^{m+1+k} \\
&=
\sum_{n=0}^{m} C_nz^n
+
z^{m+1}v_m
\left(
\sum_{k\ge 0}(zT_m)^k
\right)
\mathbf{1} \\
&=
\sum_{n=0}^{m} C_nz^n
+
z^{m+1}v_m(I-zT_m)^{-1}\mathbf{1}.
\end{align*}

By the adjugate formula,
\[
(I-zT_m)^{-1}
=
\frac{\operatorname{adj}(I-zT_m)}
{\det(I-zT_m)}.
\]
Since \(T_m\) has integer entries, \(\operatorname{adj}(I-zT_m)\) has entries in \(\mathbb{Z}[z]\), while
\[
D_m(z):=\det(I-zT_m)
\]
belongs to \(\mathbb{Z}[z]\) and satisfies \(D_m(0)=1\). It follows that
\[
A_m(z)\in\mathbb{Q}(z).
\]

Write
\[
D_m(z)=1+d_1z+\cdots+d_rz^r.
\]
The displayed formula for \(A_m(z)\) implies that
\[
D_m(z)A_m(z)=P_m(z)
\]
for some polynomial \(P_m(z)\in\mathbb{Z}[z]\). Comparing coefficients of \(z^n\) for \(n>\deg P_m\) gives
\[
a_{n,m}
+d_1a_{n-1,m}
+\cdots
+d_ra_{n-r,m}
=0.
\]
Thus \((a_{n,m})_{n\ge0}\) eventually satisfies a homogeneous constant-coefficient linear recurrence with integer coefficients.
\end{proof}

Nothing in these objects depends on the later permutation length $n$. Increasing $n$ simply means applying the same transition matrix again. This is the main enumerative payoff of the structural work developed so far. The next section illustrates this finite-dimensional description in the cases \(m=1,2,3\), while Section~\ref{sec:spectral} studies the resulting asymptotic growth.

\section{Concrete enumerations}
\label{sec:concrete}

The preceding sections give a complete finite-state procedure for every fixed adjacency radius $m$. We now carry out that procedure for the first few values of $m$. These examples serve two purposes. First, they make the abstract state construction concrete: one can see explicitly how active profiles become vertices, how insertion becomes a transition, and how the transfer matrix counts permutations. Second, they illustrate concretely the general rationality and recurrence result established in Corollary~\ref{cor:rational-gen}. In both nontrivial examples, the transfer matrix can be simplified to an explicit rational generating function and an eventual constant-coefficient linear recurrence.

We begin with $m=1$, where the family is essentially rigid, and then treat $m=2$ in detail. The case $m=3$ follows the same procedure but already produces a substantially larger state system.

\subsection{The case $m=1$}
\label{subsec:m1}

The smallest adjacency radius gives a useful boundary case. If
\[
|\pi_{i+1}-\pi_i|\le1
\]
for every $i$, then consecutive entries must differ by exactly $1$, since the entries of a permutation are distinct. Thus the permutation must traverse the values $1,\ldots,n$ consecutively in one of the two possible directions.

\begin{prop}[The rigid case $m=1$]
\label{prop:m1}
For every $n\ge3$,
\[
a_{n,1}=1.
\]
\end{prop}

\begin{proof}
The only permutations satisfying the adjacency condition are
\[
1,2,\ldots,n
\qquad\text{and}\qquad
n,n-1,\ldots,1.
\]
For $n\ge3$, the increasing permutation contains a $123$ pattern, whereas the decreasing permutation avoids $123$. Hence exactly one permutation belongs to $\mathcal C_{n,1}$.
\end{proof}

Thus $m=1$ leaves no asymptotic freedom at all. The first genuinely nontrivial finite-state system occurs at $m=2$.

\subsection{The case $m=2$}
\label{subsec:m2}

By Corollary~\ref{cor:state-count},
\[
|\Omega_2|
=
2^3-2-1
=
5.
\]
The realizable states are
\[
\Omega_2
=
\{
\dagger,\,
(0),\,
(1),\,
(0,1),\,
(1,0)
\}.
\]

This is small enough that every transition can be computed directly from the profile maps. For example, consider the state $(0)$. Front insertion adds a new deficit $0$ and increases the old deficit by $1$, giving
\[
(0)\longmapsto (0,1).
\]
Thus
\[
\Phi_{2,0}(0)=(0,1).
\]

For the state $(1)$, the raw transformed profile is
\[
(0,2).
\]
The deficit $2=m$ is no longer high enough to remain relevant, so trimming gives
\[
\Phi_{2,0}(1)=(0).
\]

The state $(1,0)$ has two active sites. Front insertion produces
\[
(0,2,1),
\]
which trims at the first $2$ to give
\[
(0).
\]
Insertion at the second active site gives
\[
(2,0),
\]
whose first coordinate is already equal to $m$, so the entire profile is trimmed away:
\[
\Phi_{2,1}(1,0)=\dagger.
\]

Carrying this out for all five states gives the complete transition rule.

\begin{prop}[Transitions for $m=2$]
\label{prop:m2-transitions}
The profile transitions are
\[
\begin{array}{c|c}
\text{parent state} & \text{child states}\\
\hline
\dagger & \varnothing\\
(0) & (0,1)\\
(1) & (0)\\
(0,1) & (0,1),\ (1,0)\\
(1,0) & \dagger,\ (0).
\end{array}
\]
\end{prop}

Ordering the states as
\[
\dagger,\ (0),\ (1),\ (0,1),\ (1,0),
\]
the transfer matrix is thus
\[
T_2
=
\begin{pmatrix}
0&0&0&0&0\\
0&0&0&1&0\\
0&1&0&0&0\\
0&0&0&1&1\\
1&1&0&0&0
\end{pmatrix}.
\]

The stable insertion process begins at length
\[
m+1=3.
\]
At this level, each of the five states occurs exactly once, so
\[
v_2
=
u_{3,2}
=
\begin{pmatrix}
1&1&1&1&1
\end{pmatrix}.
\]

The finite-state enumeration theorem now gives
\[
u_{n,2}
=
v_2T_2^{\,n-3},
\qquad n\ge3,
\]
and hence
\[
a_{n,2}
=
v_2T_2^{\,n-3}\mathbf 1.
\]

The first values are
\[
1,1,2,5,6,8,13,20,28,41,61,89,130,191,280,410,\ldots,
\]
where the first term corresponds to $n=0$.

It is worth pausing here. Nothing in this sequence was guessed from its first few terms. Once $\Omega_2$, $T_2$, and $v_2$ are known, every later value is forced by repeated application of the same five-state transition system. The matrix description also allows us to package the entire sequence into a generating function via the standard transfer-matrix method~\cite{Stanley2012}.

\begin{prop}[Generating function for $m=2$]
\label{prop:gf-m2}
Let
\[
A_2(z)
=
\sum_{n\ge0}a_{n,2}z^n.
\]
Then
\[
A_2(z)
=
1+z+2z^2
+
z^3v_2(I-zT_2)^{-1}\mathbf 1,
\]
and after simplification,
\[
A_2(z)
=
\frac{1+z^2+2z^3+z^7}
     {1-z-z^3}.
\]
\end{prop}

The denominator immediately yields a recurrence.

\begin{cor}[Recurrence for $m=2$]
\label{cor:recurrence-m2}
For $n\ge8$,
\[
a_{n,2}
=
a_{n-1,2}
+
a_{n-3,2}.
\]
\end{cor}

Thus a five-state insertion process ultimately collapses to a recurrence involving only two earlier terms. The case $m=2$ illustrates the entire finite-state method in its simplest nontrivial form:
\[
\text{profiles}
\longrightarrow
\text{transitions}
\longrightarrow
T_2
\longrightarrow
\text{exact counts}
\longrightarrow
A_2(z)
\longrightarrow
\text{recurrence}.
\]
For $m=3$, the same procedure applies, but the state space is already large enough that it is more efficient to present the resulting transition table than to derive every edge individually.

\subsection{The case $m=3$}
\label{subsec:m3}

For $m=3$, Corollary~\ref{cor:state-count} gives
\[
|\Omega_3|
=
2^4-3-1
=
12.
\]
We order the states as
\[
\begin{aligned}
\Omega_3
=
\{&
\dagger,\,
(0),\,
(1),\,
(2),\,
(0,1),\,
(0,2),\\
&
(1,0),\,
(1,2),\,
(2,0),\,
(0,1,2),\,
(0,2,1),\,
(1,2,0)
\}.
\end{aligned}
\]

Applying the transition maps from Definition~\ref{def:phi} gives the following table.

\begin{prop}[Transitions for $m=3$]
\label{prop:m3-transitions}
The child profiles, listed in active-site order, are
\[
\begin{array}{c|l}
\text{parent state} & \text{child states}\\
\hline
\dagger
    & \varnothing\\
(0)
    & (0,1)\\
(1)
    & (0,2)\\
(2)
    & (0)\\
(0,1)
    & (0,1,2),\ (1,0)\\
(0,2)
    & (0,1),\ (1,0)\\
(1,0)
    & (0,2,1),\ (2,0)\\
(1,2)
    & (0,2),\ (2,0)\\
(2,0)
    & (0),\ \dagger\\
(0,1,2)
    & (0,1,2),\ (1,0),\ (1,2,0)\\
(0,2,1)
    & (0,1),\ (1,0),\ (1)\\
(1,2,0)
    & (0,2),\ (2,0),\ (2).
\end{array}
\]
\end{prop}

The matrix $T_3$ is the $12\times12$ $0$--$1$ matrix determined by this table.

At the initial level
\[
m+1=4,
\]
the explicit initial-vector formula gives
\[
v_3
=
u_{4,3}
=
(1, 1, 1, 1, 1, 1, 2, 1, 2, 1, 1, 1)
\]
As required,
\[
v_3\mathbf 1
=
14
=
C_4.
\]

Unlike the case $m=2$, the initial multiplicities are no longer all equal to $1$. The states $(1,0)$ and $(2,0)$ each occur twice. Thus $m=3$ is the first example in which the nontrivial part of the initial-vector formula from Theorem~\ref{thm:initial-vector} is genuinely needed.

The finite-state formula gives
\[
a_{n,3}
=
v_3T_3^{\,n-4}\mathbf 1,
\qquad n\ge4.
\]

The resulting counting sequence begins
\[
\begin{aligned}
1,\,
1,\,
2,\,
5,\,
14,\,
26,\,
48,\,
91,\,
188,\,
379,\,
755,\,
1484,\,
2949,\,
5870,\,
11699,\,
23248,\,
46201,\,
91826,\ldots
\end{aligned}
\]
with the first term corresponding to $n=0$. Again, the initial terms agree with the Catalan numbers until the adjacency restriction first becomes nonvacuous. Since $m=3$, every permutation of length at most $4$ automatically satisfies the adjacency bound, so
\[
a_{n,3}=C_n,
\qquad
0\le n\le4.
\]
Beginning at $n=5$, the bounded-adjacency condition starts removing $123$-avoiding permutations. Again, the corresponding generating function is rational.

\begin{prop}[Generating function for $m=3$]
\label{prop:gf-m3}
Let
\[
A_3(z)
=
\sum_{n\ge0}a_{n,3}z^n.
\]
Then
\[
A_3(z)
=
1+z+2z^2+5z^3
+
z^4v_3(I-zT_3)^{-1}\mathbf 1.
\]
After exact simplification,
\[
A_3(z)
=
\frac{P_3(z)}{Q_3(z)},
\]
where
\[
\begin{aligned}
P_3(z)
={}&
1
+2z^3
+4z^4
+2z^5
-z^7
-2z^8
-2z^9\\
&\quad
-z^{10}
-z^{11}
+z^{15},
\end{aligned}
\]
and
\[
Q_3(z)
=
1-z-z^2-3z^4-2z^5+z^7+2z^9.
\]
\end{prop}

The denominator again produces an eventual linear recurrence.

\begin{cor}[Recurrence for $m=3$]
\label{cor:recurrence-m3}
For $n\ge16$,
\[
a_{n,3}
=
a_{n-1,3}
+
a_{n-2,3}
+
3a_{n-4,3}
+
2a_{n-5,3}
-
a_{n-7,3}
-
2a_{n-9,3}.
\]
\end{cor}

The recurrence is already considerably less obvious than in the case $m=2$. This illustrates one advantage of the transfer-matrix approach. There is no need to guess a recurrence from numerical data or to discover a separate recurrence by ad hoc combinatorial arguments for every $m$. The same profile rules produce $T_m$ systematically, and exact enumeration then follows from finite linear algebra.

\section{Spectral growth and asymptotics}
\label{sec:spectral}

The transfer-matrix description developed above gives exact values of $a_{n,m}$ for every fixed $m$. We now use the same description for a different purpose: rather than asking for the exact count at a particular $n$, we ask how quickly the count grows as $n$ becomes large. More broadly, exponential growth is fundamental in the theory of permutation classes, with every proper permutation class having an exponential upper bound~\cite{MarcusTardos2004}.

By Theorem~\ref{thm:finite-state-enumeration}, for
\[
n=m+1+k,
\qquad k\ge0,
\]
we have
\[
a_{m+1+k,m}
=
v_mT_m^k\mathbf 1.
\]
Thus, once $m$ is fixed, increasing the permutation length simply means applying the same matrix $T_m$ repeatedly. The asymptotic problem is thus a question about the behavior of large powers of one fixed nonnegative matrix.

Here, eigenvalues become useful. If $r$ is an eigenvector of $T_m$ with eigenvalue $\lambda$, then
\[
T_m^kr=\lambda^kr.
\]
Different eigendirections, hence, grow at different exponential rates. The eigenvalues of largest modulus are the natural candidates to determine the eventual scale of the counting sequence.

The first goal of this section is to make that statement precise. We then ask when the stronger asymptotic
\[
a_{n,m}\sim c_m\rho_m^n
\]
holds, and finally study how the growth constant itself changes as the adjacency radius $m$ increases.

\subsection{The exponential growth rate}
\label{subsec:growth-rate}

We begin with the basic quantity that measures the largest exponential rate available to a matrix.

\begin{definition}[Spectral radius]
\label{def:spectral-radius}
Let $M$ be a square matrix with complex eigenvalues $\lambda_1,\ldots,\lambda_r$. The \emph{spectral radius} of $M$ is
\[
\rho(M)
=
\max_{1\le i\le r}|\lambda_i|.
\]
For the transfer matrix $T_m$, we write
\[
\rho_m:=\rho(T_m).
\]
\end{definition}

Before identifying $\rho_m$ with the growth rate of $a_{n,m}$, it is useful to understand the general form forced on the counting sequence by the finite-dimensional system. If $T_m$ were diagonalizable, powers of $T_m$ would be built from terms of the form $\lambda^k$. A matrix need not be diagonalizable, however. Jordan blocks introduce polynomial factors in $k$, so the most general form is a finite sum of terms
$p(k)\lambda^k$~\cite{HornJohnson2013}. The next proposition records precisely this fact for the counting sequence.

\begin{prop}[Spectral form of the counting sequence]
\label{prop:spectral-expansion}
Fix $m\ge1$. There is a finite set $\Lambda_m$ of nonzero eigenvalues of
$T_m$ and, for each $\lambda\in\Lambda_m$, a polynomial
$p_{\lambda,m}(n)$ such that
\[
a_{n,m}
=
\sum_{\lambda\in\Lambda_m}
p_{\lambda,m}(n)\lambda^n
\]
for all sufficiently large $n$.
\end{prop}

\begin{proof}
Write
\[
n=m+1+k.
\]
Then
\[
a_{m+1+k,m}
=
v_mT_m^k\mathbf 1.
\]

Over $\mathbb C$, choose a Jordan decomposition
\[
T_m=SJS^{-1}.
\]
Consider a Jordan block of size $s$ associated with an eigenvalue $\lambda$:
\[
J_{\lambda,s}
=
\lambda I+N,
\]
where $N$ is nilpotent and
\[
N^s=0.
\]
Since $\lambda I$ and $N$ commute,
\[
J_{\lambda,s}^k
=
(\lambda I+N)^k
=
\sum_{r=0}^{s-1}
\binom{k}{r}
\lambda^{k-r}N^r.
\]

For fixed $r$, the binomial coefficient
\[
\binom{k}{r}
\]
is a polynomial in $k$. Thus, after multiplication on the left and right by fixed vectors, every Jordan block with nonzero eigenvalue contributes a polynomial in $k$ multiplied by $\lambda^k$.

Jordan blocks belonging to the eigenvalue $0$ are nilpotent, so their contribution vanishes once $k$ is sufficiently large. Therefore
\[
a_{m+1+k,m}
=
\sum_{\lambda\ne0}
q_{\lambda,m}(k)\lambda^k
\]
for suitable polynomials $q_{\lambda,m}$.

Finally,
\[
k=n-m-1
\]
differs from $n$ by only a fixed shift. Absorbing this shift into the polynomials and constant powers of $\lambda$ gives
\[
a_{n,m}
=
\sum_{\lambda\in\Lambda_m}
p_{\lambda,m}(n)\lambda^n
\]
for all sufficiently large $n$.
\end{proof}

Proposition~\ref{prop:spectral-expansion} explains why eigenvalues appear in the asymptotics, but it does not yet identify the actual exponential growth rate. Polynomial factors and several competing eigenvalues may all occur.

To isolate only the exponential scale, we take $n$th roots. If, for example,
\[
a_n\approx n^r\alpha^n,
\]
then
\[
a_n^{1/n}
\approx
n^{r/n}\alpha
\longrightarrow
\alpha.
\]
Thus an $n$th-root limit ignores polynomial factors and records only the exponential base.

\begin{theorem}[Exponential growth rate]
\label{thm:exponential-growth}
For every fixed $m\ge1$,
\[
\lim_{n\to\infty}a_{n,m}^{1/n}
=
\rho_m.
\]
\end{theorem}

\begin{proof}
Again write
\[
n=m+1+k.
\]
Then
\[
a_{m+1+k,m}
=
v_mT_m^k\mathbf 1.
\]

We first show that the counting sequence cannot grow faster than $\rho_m^k$. For any submultiplicative matrix norm,
\[
a_{m+1+k,m}
\le
\|v_m\|\,\|T_m^k\|\,\|\mathbf 1\|.
\]
Taking $k$th roots gives
\[
a_{m+1+k,m}^{1/k}
\le
\|v_m\|^{1/k}
\|T_m^k\|^{1/k}
\|\mathbf 1\|^{1/k}.
\]
The two fixed-vector factors tend to $1$, while the Gelfand formula gives
\[
\lim_{k\to\infty}
\|T_m^k\|^{1/k}
=
\rho(T_m)
=
\rho_m
\]~\cite{HornJohnson2013}.
Hence
\[
\limsup_{k\to\infty}
a_{m+1+k,m}^{1/k}
\le
\rho_m.
\]

For the reverse inequality, we must show that the part of $T_m$ responsible for the spectral radius actually contributes to the counting sequence.

Put $T_m$ into Frobenius normal form. Equivalently, decompose its directed state graph into strongly connected components. Among the corresponding diagonal blocks, choose an irreducible block $B$ satisfying
\[
\rho(B)=\rho_m.
\]

Let $x$ be the restriction of $v_m$ to the coordinates of this block, and let $y$ be the corresponding restriction of $\mathbf 1$. Every coordinate of $v_m$ is positive by Corollary~\ref{cor:initial-positive}, so
\[
x>0,
\qquad
y>0.
\]

The full quantity
\[
v_mT_m^k\mathbf 1
\]
counts all length-$k$ state walks arising from the initial distribution. In particular, it counts those walks that begin in the component represented by $B$ and remain there for all $k$ transitions. Since all matrix entries are nonnegative,
\[
a_{m+1+k,m}
\ge
xB^ky.
\]

For an irreducible nonnegative matrix $B$ and positive vectors $x,y$, Perron--Frobenius theory gives
\[
\lim_{k\to\infty}
(xB^ky)^{1/k}
=
\rho(B)
\]~\cite{HornJohnson2013}.
Therefore
\[
\liminf_{k\to\infty}
a_{m+1+k,m}^{1/k}
\ge
\rho(B)
=
\rho_m.
\]

Combining the upper and lower bounds,
\[
\lim_{k\to\infty}
a_{m+1+k,m}^{1/k}
=
\rho_m.
\]

Finally, replacing $k$ by
\[
n-m-1
\]
does not change the $n$th-root limit, since the difference is fixed. Therefore
\[
\lim_{n\to\infty}
a_{n,m}^{1/n}
=
\rho_m.
\]
\end{proof}

The lower-bound argument explains why the positivity of the initial vector was useful beyond exact enumeration. Every realizable state is already present at level $m+1$, so no dominant part of the transition graph is invisible to the counting process.

\begin{definition}[Exponential growth constant]
\label{def:growth-constant}
For fixed $m$, the quantity
\[
\rho_m
\]
will be called the \emph{exponential growth constant} of the family $(\mathcal C_{n,m})_{n\ge0}$.
\end{definition}

Thus
\[
a_{n,m}
=
\rho_m^{\,n+o(n)}.
\]
This notation should not be confused with the stronger statement
\[
a_{n,m}\sim c_m\rho_m^n,
\]
which will require additional structure. There is also an immediate algebraic consequence.

\begin{cor}[Algebraicity of the growth constant]
\label{cor:algebraic-growth}
For every fixed $m\ge1$, the growth constant $\rho_m$ is a nonnegative
algebraic integer.
\end{cor}

\begin{proof}
Because $T_m$ is a nonnegative integer matrix, $\rho_m$ is the Perron eigenvalue of an irreducible diagonal block of $T_m$. That block has integer entries, so $\rho_m$ is a root of its monic characteristic polynomial with integer coefficients. Hence $\rho_m$ is an algebraic integer. Its nonnegativity follows from Perron--Frobenius theory.
\end{proof}

\subsection{When a single exponential term dominates}
\label{subsec:pure-asymptotic}

Theorem~\ref{thm:exponential-growth} determines the exponential scale of the counting sequence for every fixed $m$. It does \emph{not} by itself imply that
\[
a_{n,m}\sim c_m\rho_m^n.
\]

The distinction matters. Several eigenvalues may have modulus $\rho_m$, which can produce periodic fluctuations. A nontrivial Jordan block at the dominant eigenvalue can introduce a polynomial factor such as $n\rho_m^n$. More generally, the structure of the dominant part of the transition graph determines whether one clean exponential term survives. To state a convenient sufficient condition, we formalize the relevant graph language.

\begin{definition}[Profile transition graph]
\label{def:profile-transition-graph}
For fixed $m$, the \emph{profile transition graph} $G_m$ is the directed graph with vertex set $\Omega_m$ and an edge
\[
\omega\longrightarrow\eta
\]
whenever
\[
(T_m)_{\omega,\eta}=1.
\]
Equivalently, there is an edge whenever insertion at some active site of a state $\omega$ produces the state $\eta$.
\end{definition}

Because the entries of $T_m$ are $0$ or $1$, the transfer matrix is exactly the adjacency matrix of $G_m$.

\begin{definition}[Strongly connected component]
\label{def:scc}
A set $K$ of vertices in a directed graph is \emph{strongly connected} if, for every $\omega,\eta\in K$, there is a directed path from $\omega$ to $\eta$ and a directed path from $\eta$ to $\omega$. A \emph{strongly connected component} is a maximal strongly connected set of vertices.
\end{definition}

Intuitively, once the state process enters a strongly connected component, every state in that component can eventually communicate with every other state in it. Consequently, such components describe the recurrent pieces of the transition system. One further condition rules out forced periodic oscillation.

\begin{definition}[Primitive component]
\label{def:primitive}
A strongly connected component is called \emph{primitive} if the adjacency matrix $B$ of that component is primitive; that is, if there exists some integer $r\ge1$ such that
\[
B^r
\]
has all entries strictly positive. Equivalently, the component is strongly connected and aperiodic.
\end{definition}

The condition says that, after sufficiently many transitions, there is no fixed congruence obstruction controlling which states can communicate. Under this hypothesis, Perron--Frobenius theory gives one genuinely dominant eigenvalue.

\begin{prop}[Primitive-core criterion]
\label{prop:primitive-core}
Suppose that $G_m$ has a unique strongly connected component $K_m$ whose adjacency matrix has spectral radius $\rho_m$, and suppose that this component is primitive.

Then there exist constants
\[
c_m>0
\qquad\text{and}\qquad
0\le\theta_m<\rho_m
\]
such that
\[
a_{n,m}
=
c_m\rho_m^n
+
O(\theta_m^n).
\]
\end{prop}

\begin{proof}
Let $B_m$ be the adjacency matrix of the unique dominant component $K_m$.

Because $B_m$ is primitive, the Perron--Frobenius theorem implies that $\rho_m$ is a simple eigenvalue of $B_m$, that it has positive left and right eigenvectors, and that every other eigenvalue of $B_m$ has modulus strictly smaller than $\rho_m$. Every other strongly connected component of $G_m$ has spectral radius strictly less than $\rho_m$ by hypothesis. Since the eigenvalues of $T_m$ are obtained from the diagonal blocks in Frobenius normal form, it follows that $\rho_m$ is a simple eigenvalue of $T_m$ and is the unique eigenvalue of $T_m$ having modulus $\rho_m$.

Hence there exist left and right eigenvectors $\ell_m$ and $r_m$ satisfying
\[
\ell_m^{T}T_m
=
\rho_m\ell_m^{T},
\qquad
T_mr_m
=
\rho_mr_m,
\]
which we normalize so that
\[
\ell_m^{T}r_m=1.
\]

The spectral decomposition then gives
\[
T_m^k
=
\rho_m^k r_m\ell_m^{T}
+
O(\theta_m^k)
\]
for some
\[
0\le\theta_m<\rho_m.
\]
Multiplying on the left by $v_m$ and on the right by $\mathbf 1$ yields
\[
a_{m+1+k,m}
=
(v_mr_m)(\ell_m^{T}\mathbf 1)\rho_m^k
+
O(\theta_m^k).
\]

The leading coefficient is positive because the initial vector is positive on the dominant component and the Perron eigenvectors are positive there. Absorbing the fixed shift $m+1$ into the leading constant gives
\[
a_{n,m}
=
c_m\rho_m^n
+
O(\theta_m^n)
\]
for some $c_m>0$.
\end{proof}

The proposition says that if there is one dominant recurrent part of the state graph and that component has no periodic obstruction, then every other contribution becomes exponentially smaller. The long-run count is thus controlled by one term.

\begin{cor}[Pure exponential asymptotic]
\label{cor:pure-asymptotic}
Under the hypotheses of Proposition~\ref{prop:primitive-core},
\[
a_{n,m}\sim c_m\rho_m^n.
\]
\end{cor}

\begin{cor}[Ratio limit]
\label{cor:ratio-limit}
Under the same hypotheses,
\[
\frac{a_{n+1,m}}{a_{n,m}}
\longrightarrow
\rho_m.
\]
\end{cor}

\begin{proof}
By Proposition~\ref{prop:primitive-core},
\[
a_{n,m}
=
c_m\rho_m^n
\left(
1+
O\left(
\left(\frac{\theta_m}{\rho_m}\right)^n
\right)
\right).
\]
Since
\[
\frac{\theta_m}{\rho_m}<1,
\]
the error tends to zero geometrically. Therefore
\[
\frac{a_{n+1,m}}{a_{n,m}}
\longrightarrow
\rho_m.
\]
\end{proof}

The same argument says more than how many permutations there are. It also describes how those permutations are distributed among the active-profile states.

For a uniformly chosen permutation
\[
\pi\in\mathcal C_{n,m},
\]
we have
\[
\mathbb P\bigl(\lambda_m(\pi)=\omega\bigr)
=
\frac{u_{n,m}(\omega)}{a_{n,m}}.
\]
Thus normalizing the state-count vector turns it into a probability distribution on $\Omega_m$.

\begin{cor}[Limiting profile distribution]
\label{cor:profile-distribution}
Under the hypotheses of Proposition~\ref{prop:primitive-core},
\[
\frac{u_{n,m}}{a_{n,m}}
\longrightarrow
\frac{\ell_m^T}{\ell_m^T\mathbf 1},
\]
where $\ell_m^T$ is a nonnegative left Perron eigenvector associated with
$\rho_m$.

Equivalently, for every state $\omega\in\Omega_m$,
\[
\frac{u_{n,m}(\omega)}{a_{n,m}}
\]
converges to the corresponding coordinate of the normalized left Perron eigenvector.
\end{cor}

\begin{proof}
Write
\[
n=m+1+k.
\]
Under the hypotheses of Proposition~\ref{prop:primitive-core},
\[
T_m^k
=
\rho_m^k r_m\ell_m^T
+
O(\theta_m^k).
\]
Hence
\[
u_{n,m}
=
v_mT_m^k
=
(v_mr_m)\rho_m^k\ell_m^T
+
O(\theta_m^k).
\]
Summing the coordinates gives
\[
a_{n,m}
=
(v_mr_m)(\ell_m^T\mathbf 1)\rho_m^k
+
O(\theta_m^k).
\]
Dividing the vector relation by the scalar relation yields
\[
\frac{u_{n,m}}{a_{n,m}}
\longrightarrow
\frac{\ell_m^T}{\ell_m^T\mathbf 1}.
\]
\end{proof}

Thus the Perron eigenvector has a direct combinatorial interpretation. It describes the limiting frequency of the different extension types among large admissible permutations.

\begin{remark}[Spectral gap and convergence rate]
\label{rem:spectral-gap}
Under the primitive-core hypothesis, let
\[
\theta_m
=
\max\{
|\lambda|:
\lambda\in\operatorname{Spec}(T_m),
\ \lambda\ne\rho_m
\}.
\]
Then
\[
\theta_m<\rho_m.
\]
After normalizing by $\rho_m^n$, the remaining spectral contributions decay at an exponential rate governed by
\[
\frac{\theta_m}{\rho_m}.
\]
Thus, a large gap between $\rho_m$ and the remaining eigenvalues means that the asymptotic behavior becomes visible quickly, while a small gap means
that finite-$n$ behavior can remain far from its limiting form for much longer.
\end{remark}

The concrete cases $m=2$ and $m=3$ satisfy the primitive-core condition. Their recurrent components are strongly connected and contain self-loops, so the corresponding adjacency matrices are primitive. Consequently, the growth constants determined by the corresponding transfer matrices control not only the $n$th-root growth but the full leading asymptotic in those cases.

\subsection{Dependence of the growth rate on $m$}
\label{subsec:growth-vs-m}

Everything so far has fixed $m$ and allowed the permutation length $n$ to grow. We now ask a different question: how does the resulting growth constant $\rho_m$ change when the adjacency radius itself increases?

There is an obvious benchmark. Without the adjacency restriction, $123$-avoiding permutations are counted by the Catalan numbers,
\[
C_n
=
\frac{1}{n+1}\binom{2n}{n},
\]
and
\[
C_n^{1/n}\longrightarrow4.
\]
Thus $4$ is the unrestricted exponential growth constant.

Since increasing $m$ weakens the adjacency restriction, one should expect $\rho_m$ to increase toward $4$. The next theorem confirms this and also gives a quantitative lower bound.

\begin{theorem}[Bounds and limiting growth rate]
\label{thm:rho-bounds}
For every $m\ge1$,
\[
C_{m-1}^{\,1/(m+1)}
\le
\rho_m
<
4.
\]
Moreover,
\[
\rho_m\le\rho_{m+1}
\]
for every $m$, and
\[
\rho_m\longrightarrow4
\qquad\text{as }m\to\infty.
\]
More quantitatively,
\[
0<4-\rho_m
=
O\left(\frac{\log m}{m}\right).
\]
\end{theorem}

\begin{proof}
We begin with monotonicity. Increasing the adjacency radius can only weaken the restriction, so
\[
\mathcal C_{n,m}
\subseteq
\mathcal C_{n,m+1}.
\]
Hence
\[
a_{n,m}
\le
a_{n,m+1}
\]
for every $n$. Taking $n$th roots and then passing to the limit gives
\[
\rho_m
\le
\rho_{m+1}.
\]

Since
\[
\mathcal C_{n,m}
\subseteq
\operatorname{Av}_n(123),
\]
we also have
\[
a_{n,m}\le C_n.
\]
Therefore
\[
\rho_m\le4.
\]

We claim that equality cannot hold for any fixed $m$.

Let
\[
A_m(z)
=
\sum_{n\ge0}a_{n,m}z^n
\]
and let
\[
C(z)
=
\sum_{n\ge0}C_nz^n
=
\frac{1-\sqrt{1-4z}}{2z}
\]
be the Catalan generating function.

Coefficient-wise,
\[
0\le A_m(z)\le C(z)
\]
for real
\[
0\le z<\frac14.
\]
Moreover,
\[
C(z)\longrightarrow2
\qquad
\text{as }
z\uparrow\frac14.
\]
Thus $A_m(z)$ remains bounded as $z$ approaches $1/4$ from below.

Suppose for contradiction that
\[
\rho_m=4.
\]
By Theorem~\ref{thm:exponential-growth}, the radius of convergence of
$A_m(z)$ would then be
\[
\frac14.
\]
Since $A_m(z)$ has nonnegative coefficients, Pringsheim's theorem implies that the positive real point
\[
z=\frac14
\]
is a singularity of $A_m$~\cite{FlajoletSedgewick2009}.

But $A_m(z)$ is rational by the generating-function results of Corollary~\ref{cor:rational-gen}. Any finite singularity of a rational function is a pole. Therefore, $A_m(z)$ would be unbounded as
\[
z\uparrow\frac14,
\]
contradicting
\[
0\le A_m(z)\le C(z)\longrightarrow2.
\]
Hence
\[
\rho_m<4
\]
for every fixed $m$.

It remains to obtain a lower bound approaching $4$. We adapt the descending-block construction of Mayama and Akita~\cite{MayamaAkita2026}.

Fix $m$, and let
\[
n=q(m+1).
\]
Partition the values
\[
1,2,\ldots,n
\]
into $q$ consecutive blocks of size $m+1$, and arrange the blocks from largest to smallest. Thus every entry in an earlier block is larger than every entry in a later block. Within each block, place the largest value first and the smallest value last. Arrange the remaining $m-1$ values according to an arbitrary $123$ avoiding permutation. Schematically, each block has the form
\[
\text{maximum}
\quad
\underbrace{
\text{$123$-avoider on the middle $m-1$ values}
}_{C_{m-1}\text{ choices}}
\quad
\text{minimum}.
\]

Every pair of values within one block differs by at most $m$, so the adjacency restriction is automatic inside a block. At the boundary between two consecutive blocks, the final entry of the higher block is its minimum, while the first entry of the lower block is its maximum. Because the blocks are consecutive intervals, these two values differ by exactly $1$. The resulting permutation also avoids $123$. A $123$ pattern cannot use entries from two different blocks, because every entry in an earlier block is larger than every entry in a later block. Within a single block, the first maximum and final minimum cannot participate in an increasing triple, so any $123$ pattern would have to lie entirely inside the middle portion. That middle portion was chosen to avoid $123$. Hence, every independent choice of a middle $123$-avoider in each block produces a permutation in $\mathcal C_{q(m+1),m}$. There are
\[
C_{m-1}^q
\]
such choices, so
\[
a_{q(m+1),m}
\ge
C_{m-1}^q.
\]

Taking $(q(m+1))$th roots gives
\[
a_{q(m+1),m}^{1/(q(m+1))}
\ge
C_{m-1}^{1/(m+1)}.
\]
Letting $q\to\infty$ and applying
Theorem~\ref{thm:exponential-growth},
\[
\rho_m
\ge
C_{m-1}^{1/(m+1)}.
\]

We, as a result, have
\[
C_{m-1}^{1/(m+1)}
\le
\rho_m
<
4.
\]

The Catalan asymptotic
\[
C_{m-1}
\sim
\frac{4^{m-1}}
{\sqrt{\pi}(m-1)^{3/2}}
\]
implies
\[
C_{m-1}^{1/(m+1)}
\longrightarrow4
\]~\cite{FlajoletSedgewick2009}.
Hence, the squeeze theorem gives
\[
\rho_m\longrightarrow4.
\]

Finally,
\[
\log
\left(
\frac{4}
{C_{m-1}^{1/(m+1)}}
\right)
=
\log4
-
\frac{\log C_{m-1}}{m+1}
=
O\left(\frac{\log m}{m}\right).
\]
Write
\[
C_{m-1}^{1/(m+1)}
=
4e^{-x_m},
\]
where
\[
x_m
=
O\left(\frac{\log m}{m}\right).
\]
Since
\[
1-e^{-x}\le x
\qquad
(x\ge0),
\]
we obtain
\[
0
\le
4-\rho_m
\le
4-C_{m-1}^{1/(m+1)}
=
4(1-e^{-x_m})
=
O\left(\frac{\log m}{m}\right).
\]
Since $\rho_m<4$ for every finite $m$, we have that
\[
4-\rho_m > 0
\]
\end{proof}

The theorem gives a useful picture of the effect of the adjacency restriction. For every fixed finite $m$, it removes enough permutations to lower the exponential growth rate strictly below the unrestricted Catalan value $4$. Nevertheless, as the permitted jump size increases, the loss becomes negligible on the exponential scale:
\[
\rho_m\uparrow4.
\]

The lower bound also shows why this happens. Large values of $m$ allow increasingly large blocks inside which essentially unrestricted Catalan behavior can occur, while arranging those blocks in decreasing value order prevents interactions between them from creating a $123$ pattern.

The preceding theorem takes the limits in the order
\[
n\to\infty
\quad\text{first, with $m$ fixed,}
\qquad
m\to\infty
\quad\text{afterward.}
\]
A related construction shows that the same exponential scale is recovered even when the adjacency radius grows together with the permutation length.

\begin{cor}[Growing adjacency radius]
\label{cor:growing-radius}
Let $m=m(n)$ be an integer-valued function satisfying
\[
m(n)\longrightarrow\infty.
\]
Then
\[
a_{n,m(n)}^{1/n}
\longrightarrow4.
\]
\end{cor}

\begin{proof}
The Catalan upper bound gives
\[
a_{n,m(n)}
\le
C_n,
\]
and thus
\[
\limsup_{n\to\infty}
a_{n,m(n)}^{1/n}
\le
4.
\]

For the lower bound, set
\[
L_n
=
\min\{m(n)+1,n\}.
\]
Then
\[
L_n\longrightarrow\infty.
\]

Partition the values \(1,\ldots,n\) into
\[
q_n=\left\lceil\frac{n}{L_n}\right\rceil
\]
consecutive blocks, arranged from largest to smallest, with block sizes differing by at most one. Write these sizes as
\[
b_{1,n},\ldots,b_{q_n,n}.
\]
Thus each block has size either
\[
\left\lfloor\frac{n}{q_n}\right\rfloor
\qquad\text{or}\qquad
\left\lceil\frac{n}{q_n}\right\rceil.
\]

Since
\[
q_n\ge \frac{n}{L_n},
\]
we have
\[
\frac{n}{q_n}\le L_n,
\]
and, as a result, 
\[
b_{i,n}\le L_n\le m(n)+1
\]
for every \(i\). On the other hand,
\[
q_n
\le
\frac{n}{L_n}+1,
\]
so, using \(L_n\le n\),
\[
\frac{n}{q_n}
\ge
\frac{n}{n/L_n+1}
=
\frac{nL_n}{n+L_n}
\ge
\frac{L_n}{2}.
\]
Consequently,
\[
\min_i b_{i,n}
=
\left\lfloor\frac{n}{q_n}\right\rfloor
\ge
\left\lfloor\frac{L_n}{2}\right\rfloor
\longrightarrow\infty.
\]

Within a block of size $b_{i,n}$, place its maximum first and its minimum last, and arrange the remaining
\[
b_{i,n}-2
\]
values according to an arbitrary $123$-avoiding permutation. The same argument used in Theorem~\ref{thm:rho-bounds} shows that concatenating these blocks produces a valid member of $\mathcal C_{n,m(n)}$.

Hence
\[
a_{n,m(n)}
\ge
\prod_{i=1}^{q_n}
C_{b_{i,n}-2}.
\]

Using
\[
\log C_r
=
r\log4-O(\log r),
\]
we obtain
\[
\begin{aligned}
\log a_{n,m(n)}
&\ge
\sum_{i=1}^{q_n}
\left[
(b_{i,n}-2)\log4
-
O(\log b_{i,n})
\right]\\
&=
n\log4
-
O(q_n\log L_n).
\end{aligned}
\]
Since
\[
q_n
=
O\left(\frac{n}{L_n}\right),
\]
it follows that
\[
\frac1n\log a_{n,m(n)}
\ge
\log4
-
O\left(
\frac{\log L_n}{L_n}
\right).
\]
Because
\[
L_n\longrightarrow\infty,
\]
the error term tends to zero. Thus
\[
\liminf_{n\to\infty}
a_{n,m(n)}^{1/n}
\ge
4.
\]

Combining this with the Catalan upper bound gives
\[
a_{n,m(n)}^{1/n}
\longrightarrow4.
\]
\end{proof}

The fixed-$m$ and growing-$m$ results together describe the exponential scale of the family across the two natural regimes. A fixed adjacency radius produces a growth constant
\[
\rho_m<4,
\]
determined by the spectral radius of a finite transition matrix. Increasing the radius enlarges the family and increases this constant, and any adjacency radius tending to infinity eventually recovers the unrestricted Catalan growth rate on the exponential scale.

\section{Pattern symmetries under the adjacency constraint}
\label{sec:symmetries}

For ordinary pattern avoidance, all six patterns of length three have the same enumeration:
\[
|\operatorname{Av}_n(\tau)|=C_n
\qquad
\text{for every }\tau\in S_3
\]~\cite{SimionSchmidt1985}.
It is natural to ask whether the same equivalence survives after the bounded-adjacency condition is imposed. There is a reason to be cautious. The condition
\[
|\pi_{i+1}-\pi_i|\le m
\]
is not purely a condition on the relative order of the entries. It couples two different structures: consecutive \emph{positions} must contain numerically nearby \emph{values}. A symmetry that changes positions or values in a distance-preserving way may survive, while one that interchanges the roles of positions and values need not.

The purpose of this section is to determine which of the classical symmetries remain valid and, consequently, which length-three patterns must still have the same constrained enumeration.

\begin{definition}[Adjacency-constrained avoidance class]
\label{def:general-constrained-class}
Let $\tau\in S_k$. Define
\[
\mathcal C_{n,m}(\tau)
=
\left\{
\pi\in S_n:
\pi\text{ avoids }\tau
\text{ and }
|\pi_{i+1}-\pi_i|\le m
\text{ for }1\le i<n
\right\}.
\]
Write
\[
a_{n,m}(\tau)
=
|\mathcal C_{n,m}(\tau)|.
\]
\end{definition}

Thus the family studied throughout the preceding sections is
\[
\mathcal C_{n,m}
=
\mathcal C_{n,m}(123),
\]
and
\[
a_{n,m}=a_{n,m}(123).
\]

\subsection{Constraint-preserving symmetries}
\label{subsec:preserved-symmetries}

We first consider reversal and complement.

\begin{definition}[Reversal, complement, and inverse]
\label{def:permutation-symmetries}
Let
\[
\pi=\pi_1\pi_2\cdots\pi_n\in S_n.
\]
Its \emph{reversal} is the permutation $R(\pi)$ defined by
\[
R(\pi)_i
=
\pi_{n+1-i}.
\]
Its \emph{complement} is the permutation $C(\pi)$ defined by
\[
C(\pi)_i
=
n+1-\pi_i.
\]
The \emph{reverse-complement} is
\[
RC(\pi)=R(C(\pi)),
\]
and the \emph{inverse} is the usual inverse permutation
\[
I(\pi)=\pi^{-1}.
\]
\end{definition}

Reversal changes the order in which adjacent pairs are read, but does not change the differences within those pairs. Complement reflects the values about the midpoint of $[n]$, which also preserves absolute differences. Thus both transformations are natural candidates to remain symmetries of the constrained problem.

\begin{prop}[Constraint-preserving pattern symmetries]
\label{prop:preserved-symmetries}
For every pattern $\tau$, and for all $n$ and $m$,
\[
a_{n,m}(\tau)
=
a_{n,m}(R(\tau))
=
a_{n,m}(C(\tau))
=
a_{n,m}(RC(\tau)).
\]
\end{prop}

\begin{proof}
Let
\[
\pi\in S_n.
\]

For reversal,
\[
R(\pi)_i=\pi_{n+1-i},
\]
so
\[
\begin{aligned}
|R(\pi)_{i+1}-R(\pi)_i|
&=
|\pi_{n-i}-\pi_{n+1-i}|.
\end{aligned}
\]
Thus the adjacent differences of $R(\pi)$ are exactly the adjacent differences of $\pi$ read in the opposite order. Hence
\[
|\pi_{i+1}-\pi_i|\le m
\quad\text{for all }i
\]
if and only if the same condition holds for $R(\pi)$.

For complement,
\[
C(\pi)_i=n+1-\pi_i,
\]
and thus
\[
\begin{aligned}
|C(\pi)_{i+1}-C(\pi)_i|
&=
|(n+1-\pi_{i+1})-(n+1-\pi_i)|\\
&=
|\pi_{i+1}-\pi_i|.
\end{aligned}
\]
Thus, complement also preserves the adjacency restriction. 

The standard correspondence for permutation patterns gives
\[
\pi\text{ avoids }\tau
\iff
R(\pi)\text{ avoids }R(\tau),
\]
and similarly for complement and reverse-complement.

Since reversal and complement are bijections on $S_n$, they restrict to bijections between the corresponding adjacency-constrained avoidance classes. Hence
\[
a_{n,m}(\tau)
=
a_{n,m}(R(\tau))
=
a_{n,m}(C(\tau))
=
a_{n,m}(RC(\tau)).
\]
\end{proof}

For patterns of length three, these symmetries immediately divide $S_3$ into two classes.

\begin{cor}[Constraint-preserving symmetry classes in $S_3$]
\label{cor:s3-symmetry-classes}
For all $n$ and $m$,
\[
a_{n,m}(123)=a_{n,m}(321)
\]
and
\[
a_{n,m}(132)
=
a_{n,m}(213)
=
a_{n,m}(231)
=
a_{n,m}(312).
\]
Thus reversal and complement produce the two symmetry orbits
\[
\{123,321\}
\qquad\text{and}\qquad
\{132,213,231,312\}.
\]
\end{cor}

In particular, every enumerative result established above for $123$ transfers immediately to $321$. The two patterns have identical counting sequences, generating functions, recurrences, and exponential growth constants under the same adjacency constraint.

The remaining standard symmetry, inversion, behaves differently. The reason is visible directly from the constraint.

The original condition
\[
|\pi_{i+1}-\pi_i|\le m
\]
says that entries occupying consecutive positions must have nearby values. For the inverse permutation, the corresponding condition is
\[
|\pi^{-1}(j+1)-\pi^{-1}(j)|\le m,
\]
which says that consecutive \emph{values} must occur in nearby \emph{positions}. These are different requirements.

\begin{prop}[Inversion does not preserve bounded adjacency]
\label{prop:inverse-fails}
The condition
\[
|\pi_{i+1}-\pi_i|\le m
\]
is not, in general, preserved by inversion.
\end{prop}

\begin{proof}
Take
\[
\pi=1342
\qquad\text{and}\qquad
m=2.
\]
The adjacent differences of $\pi$ are
\[
|3-1|=2,\qquad
|4-3|=1,\qquad
|2-4|=2,
\]
so $\pi$ satisfies the adjacency restriction.

However,
\[
\pi^{-1}=1423.
\]
Its first adjacent difference is
\[
|4-1|=3>2.
\]
Thus $\pi^{-1}$ does not satisfy the same adjacency restriction.
\end{proof}

The failure of inversion is conceptually important even though, for patterns of length three, reversal and complement already produce the two orbits in Corollary~\ref{cor:s3-symmetry-classes}. It shows that the full collection of standard permutation-plot symmetries cannot automatically be imported into the constrained setting. The adjacency condition distinguishes the roles of positions and values, and inversion exchanges those roles.

\subsection{Failure of ordinary Wilf equivalence}
\label{subsec:wilf-failure}

The symmetry classes above tell us which constrained enumerations are \emph{forced} to agree. They do not tell us whether the two classes might nevertheless have the same cardinality for some additional reason. In the unconstrained problem, they do. Every length-three avoidance class is Catalan, so
\[
|\operatorname{Av}_n(123)|
=
|\operatorname{Av}_n(132)|
=
\cdots
=
C_n.
\]
This equality is stronger than what reversal and complement alone imply. Under the adjacency restriction, that additional equivalence fails.

\begin{prop}[Failure of length-three Wilf equivalence]
\label{prop:wilf-failure}
For
\[
n=4
\qquad\text{and}\qquad
m=2,
\]
we have
\[
a_{4,2}(123)=a_{4,2}(321)=6,
\]
whereas
\[
a_{4,2}(132)
=
a_{4,2}(213)
=
a_{4,2}(231)
=
a_{4,2}(312)
=
8.
\]
Consequently, the six length-three patterns are not Wilf-equivalent under the bounded-adjacency constraint.
\end{prop}

\begin{proof}
A direct enumeration gives
\[
\mathcal C_{4,2}(123)
=
\{
2431,\,
3421,\,
4213,\,
4231,\,
4312,\,
4321
\},
\]
so
\[
a_{4,2}(123)=6.
\]

Similarly,
\[
\begin{aligned}
\mathcal C_{4,2}(132)
=
\{&
1234,\,
2134,\,
3124,\,
3421,\\
&
4213,\,
4231,\,
4312,\,
4321
\},
\end{aligned}
\]
and hence
\[
a_{4,2}(132)=8.
\]

The remaining equalities follow from Corollary~\ref{cor:s3-symmetry-classes}.
\end{proof}

The distinction is worth emphasizing. Even in ordinary pattern avoidance, the geometric symmetries generated by reversal and complement divide $S_3$ into the two orbits
\[
\{123,321\}
\qquad\text{and}\qquad
\{132,213,231,312\}.
\]
The fact that both orbits are counted by the Catalan numbers is an additional enumerative coincidence, not a consequence of those symmetries alone. The bounded-adjacency restriction preserves the equalities within each orbit but destroys the equality between the two orbits. Thus, the analysis developed in this paper for $123$ simultaneously solves the $321$ case, but it does not automatically extend to the remaining four patterns. Those patterns form a genuinely different constrained enumeration problem.

\section{Algorithmic and graph-theoretic interpretations}
\label{sec:graph-algorithmic}

The preceding sections developed the enumeration primarily in the language of permutations and active profiles. Both viewpoints also admit natural graph-theoretic interpretations. At the level of an individual permutation, the bounded-adjacency condition turns the problem into one of finding Hamiltonian paths in a power of the ordinary path graph. At the level of recursive growth, the profile transition graph $G_m$ records the possible ways in which an admissible permutation can progress when successive maxima are inserted.

These are different graphs serving different purposes. For example,  $P_n^m$ describes the admissible permutations themselves, where $G_m$ describes their possible extension histories. The first interpretation clarifies the geometry of the original constraint. The second explains why exact enumeration can be carried out by a sparse dynamic program and why the same machinery also supports exact random generation.

\subsection{A Hamiltonian-path interpretation}
\label{subsec:hamiltonian}

The adjacency condition
\[
|\pi_{i+1}-\pi_i|\le m
\]
already has the form of an edge condition: two consecutive entries are allowed to occur next to one another precisely when their numerical values differ by at most $m$. This suggests placing an edge between exactly those values. Related Hamiltonian-path formulations for permutations with bounded adjacent differences appear in Avgustinovich and Kitaev~\cite{AvgustinovichKitaev2008}. Here the additional \(123\)-avoidance condition becomes a restriction on the label order along the path.

\begin{definition}[Power of the path graph]
\label{def:path-power}
Let $P_n$ be the path graph with vertex set
\[
[n]=\{1,2,\ldots,n\}.
\]
For $m\ge1$, the \emph{$m$th power of the path graph}, denoted $P_n^m$, is the graph on the same vertex set in which two distinct vertices $u,v$ are adjacent precisely when
\[
|u-v|\le m.
\]
\end{definition}

A permutation
\[
\pi=\pi_1\pi_2\cdots\pi_n
\]
is already an ordering of the vertices of $P_n^m$. Since every value of $[n]$ appears exactly once, reading the permutation from left to right visits every vertex exactly once. The bounded-adjacency condition says precisely that each consecutive pair in this ordering is an edge. Thus one of the two defining conditions of $\mathcal C_{n,m}$ becomes a standard graph-theoretic condition.

\begin{theorem}[Hamiltonian-path interpretation]
\label{thm:hamiltonian-path}
There is a natural bijection between $\mathcal C_{n,m}$ and the set of oriented Hamiltonian paths
\[
v_1,v_2,\ldots,v_n
\]
in $P_n^m$ whose vertex-label sequence avoids $123$.

Equivalently,
\[
\mathcal C_{n,m}
\]
is the set of oriented Hamiltonian paths in $P_n^m$ whose label sequence has longest increasing subsequence of length at most $2$.
\end{theorem}

\begin{proof}
Let
\[
\pi=\pi_1\pi_2\cdots\pi_n\in S_n.
\]
View the entries of $\pi$ as an ordering of the vertices of $P_n^m$. 

Since $\pi$ is a permutation, the sequence
\[
\pi_1,\pi_2,\ldots,\pi_n
\]
visits every vertex of $[n]$ exactly once.

By Definition~\ref{def:path-power},
\[
\pi_i
\sim
\pi_{i+1}
\quad\Longleftrightarrow\quad
|\pi_{i+1}-\pi_i|\le m.
\]
Hence the bounded-adjacency condition is equivalent to saying that the vertex ordering defined by $\pi$ is an oriented Hamiltonian path in $P_n^m$.

The remaining condition is exactly the pattern condition. The word
\[
\pi_1\pi_2\cdots\pi_n
\]
avoids $123$ if and only if it contains no increasing subsequence of length $3$, which is equivalent to
\[
\operatorname{LIS}(\pi)\le2.
\]

The correspondence does not alter the underlying ordering, so it is bijective.
\end{proof}

The theorem thus gives the equivalent formulation
\[
a_{n,m}
=
\#\left\{
\begin{array}{c}
\text{oriented Hamiltonian paths in }P_n^m\\
\text{whose label sequence has }\operatorname{LIS}\le2
\end{array}
\right\}.
\]

This reformulation is useful because it separates the two restrictions in a particularly concrete way:
\[
\text{bounded adjacency}
\iff
\text{the path uses only edges of }P_n^m
\]
and
\[
123\text{-avoidance}
\iff
\operatorname{LIS}\le2\text{ along the visitation order}.
\]

Thus the original problem may be viewed as a constrained routing problem: visit every vertex exactly once, take only jumps of size at most $m$, and never allow the sequence of visited labels to contain an increasing triple. The reverse-diagonal localization theorem also has a simple interpretation in this model.

\begin{remark}[Localization as a bounded corridor]
\label{rem:graph-localization}
The completely decreasing permutation
\[
n,n-1,\ldots,1
\]
is itself an oriented Hamiltonian path in $P_n^m$ for every $m\ge1$.

At time $i$, this path visits the vertex
\[
n+1-i.
\]
The localization theorem gives
\[
|\pi_i-(n+1-i)|\le m.
\]
Hence every admissible Hamiltonian path must, at time $i$, visit a vertex lying within graph-label distance $m$ of the vertex visited by the completely decreasing traversal.

In this sense, every valid path remains inside a width-$m$ corridor around
\[
n,n-1,\ldots,1.
\]
The path may deviate from the decreasing order, but it can never wander arbitrarily far from it.
\end{remark}

The two extreme adjacency regimes are also immediate from this viewpoint. When $m=1$,
\[
P_n^1=P_n,
\]
so there are only two oriented Hamiltonian paths, namely the increasing and decreasing traversals. At the opposite extreme, if
\[
m\ge n-1,
\]
then
\[
P_n^m=K_n,
\]
so the graph places no restriction at all on the vertex ordering. The problem then reduces to ordinary $123$-avoidance. The Hamiltonian-path model describes what an individual admissible permutation looks like. We now turn to a second graph, which instead describes how admissible permutations grow.

\subsection{The profile transition graph}
\label{subsec:transition-graph-interpretation}

Recall the profile transition graph $G_m$ introduced in Definition~\ref{def:profile-transition-graph}. Its vertices are the realizable profiles
\[
\Omega_m,
\]
and there is a directed edge
\[
\omega\longrightarrow\eta
\]
whenever insertion at one of the active sites of $\omega$ produces the child profile $\eta$.

By Proposition~\ref{prop:zero-one}, distinct active sites of a fixed state produce distinct child states. Consequently, there is at most one edge between any ordered pair of states, and the transfer matrix $T_m$ is exactly the $0$--$1$ adjacency matrix of $G_m$. This gives a direct interpretation of powers of $T_m$.

\begin{prop}[Walk interpretation of the transfer matrix]
\label{prop:walk-interpretation} 
For states
\[
\omega,\eta\in\Omega_m
\]
and every $k\ge0$,
\[
(T_m^k)_{\omega,\eta}
\]
is the number of directed walks of length $k$ in $G_m$ beginning at $\omega$ and ending at $\eta$.

Moreover, if
\[
n=m+1+k,
\]
then every permutation in $\mathcal C_{n,m}$ is uniquely determined by

\begin{enumerate}
    \item an initial permutation in $\mathcal C_{m+1,m}$, and
    \item a directed walk of length $k$ in $G_m$ beginning at the active profile of that initial permutation.
\end{enumerate}
\end{prop}

\begin{proof}
The first statement is the standard interpretation of powers of an adjacency matrix.

For the second, begin with
\[
\pi^{(0)}\in\mathcal C_{m+1,m}.
\]
Each edge of a walk in $G_m$ corresponds to insertion at one active site, and, as a result, determines the next permutation by inserting the next maximum. After $k$ transitions, this produces a permutation of length
\[
m+1+k.
\]

Conversely, starting from any
\[
\pi\in\mathcal C_{m+1+k,m},
\]
repeatedly delete the current maximum. Closure under maximum deletion produces a unique chain
\[
\pi
\longrightarrow
\pi^{(k-1)}
\longrightarrow
\cdots
\longrightarrow
\pi^{(0)}
\]
ending at length $m+1$.

Thus the initial permutation and the sequence of profile transitions are uniquely recoverable from the final permutation.
\end{proof}

The transfer formula
\[
a_{m+1+k,m}
=
v_mT_m^k\mathbf 1
\]
has a purely combinatorial interpretation. The vector $v_m$ counts the possible initial permutations in each starting state, $T_m^k$ counts their possible $k$-step extension histories, and $\mathbf 1$ sums over all possible terminal states.

The graph $G_m$ is much smaller than the space of permutations it encodes. Its exact size can also be determined.

\begin{prop}[Size of the profile transition graph]
\label{prop:transition-graph-size}
The graph $G_m$ has
\[
|V(G_m)|
=
2^{m+1}-m-1
\]
vertices and
\[
|E(G_m)|
=
\frac{m+1}{2}
\left(
2^{m+1}-m-2
\right)
\]
directed edges.
\end{prop}

\begin{proof}
The vertex count is the exact state count from Corollary~\ref{cor:state-count}:
\[
|V(G_m)|
=
|\Omega_m|
=
2^{m+1}-m-1.
\]

A realizable profile of length $\ell$ has exactly $\ell$ active sites and, accordingly, exactly $\ell$ outgoing edges.

By Corollary~\ref{cor:profiles-fixed-length}, the number of realizable profiles of length $\ell$ is
\[
\binom{m+1}{\ell}-1.
\]
The dead state has no outgoing edges. Hence
\[
|E(G_m)|
=
\sum_{\ell=1}^{m}
\ell
\left(
\binom{m+1}{\ell}-1
\right).
\]

Now
\[
\sum_{\ell=0}^{m+1}
\ell
\binom{m+1}{\ell}
=
(m+1)2^m.
\]
Removing the $\ell=m+1$ term gives
\[
\sum_{\ell=1}^{m}
\ell
\binom{m+1}{\ell}
=
(m+1)2^m-(m+1).
\]
Also,
\[
\sum_{\ell=1}^{m}\ell
=
\frac{m(m+1)}{2}.
\]
Therefore
\[
\begin{aligned}
|E(G_m)|
&=
(m+1)2^m-(m+1)
-\frac{m(m+1)}{2}\\
&=
\frac{m+1}{2}
\left(
2^{m+1}-m-2
\right).
\end{aligned}
\]
\end{proof}

In particular,
\[
|V(G_m)|=O(2^m)
\]
and
\[
|E(G_m)|=O(m2^m).
\]

This sparsity is important. A generic matrix with
\[
O(2^m)
\]
rows and columns could have
\[
O(4^m)
\]
nonzero entries. The transfer matrix arising here has only
\[
O(m2^m)
\]
nonzero entries because a state has at most $m$ possible children.

The state description is thus not only finite; it is algorithmically sparse.

\subsection{Sparse enumeration and exact generation}
\label{subsec:algorithms}

The transfer relation
\[
u_{n+1,m}
=
u_{n,m}T_m
\]
could be implemented by ordinary matrix multiplication, but doing so would ignore the sparse structure of $G_m$. Instead, one may update the state counts directly along the edges of the transition graph. For each state $\omega$, its current count is added once to each child state $\eta$ satisfying
\[
\omega\longrightarrow\eta.
\]

Each edge is, consequently, processed exactly once per level.

\begin{cor}[Sparse dynamic-programming complexity]
\label{cor:sparse-dp}
Once the state space, transition rule, and initial vector have been constructed, the values
\[
a_{m+1,m},
a_{m+2,m},
\ldots,
a_{N,m}
\]
can be computed using
\[
O(Nm2^m)
\]
arithmetic operations.

If the full transition graph is stored explicitly, it requires
\[
O(m2^m)
\]
storage. If transitions are generated directly from the maps $\Phi_{m,j}$ as needed, the dynamic program requires only
\[
O(2^m)
\]
state-count entries of working storage, aside from the representation of the current state being processed.
\end{cor}

\begin{proof}
By Proposition~\ref{prop:transition-graph-size},
\[
|E(G_m)|
=
O(m2^m).
\]
One update
\[
u_{n,m}
\longmapsto
u_{n+1,m}
\]
can be performed by traversing every directed edge once and adding the count at its source to the count at its target. Hence one level requires
\[
O(m2^m)
\]
additions.

Repeating this through length $N$ gives
\[
O(Nm2^m)
\]
arithmetic operations.

Explicitly storing $G_m$ requires one record for each vertex and edge, and hence
\[
O(m2^m)
\]
storage.

Alternatively, the children of a state may be generated directly using
\[
\Phi_{m,0},
\Phi_{m,1},
\ldots,
\Phi_{m,\ell-1}.
\]
In that implementation, only the current and next state-count vectors need to be retained. Each has
\[
|\Omega_m|=O(2^m)
\]
coordinates.
\end{proof}

The complexity above is stated in arithmetic operations: an addition of two integer state counts is treated as one operation. Since the counts themselves grow exponentially with $n$, their bit lengths also grow with $n$, so a bit-complexity analysis would include the cost of operating on these larger integers. For fixed $m$, however, the important structural point is that the number of state updates grows only linearly with $N$. The same dynamic program can be run backward to count completions from each state. This gives an exact uniform sampling algorithm.

\begin{definition}[Completion count]
\label{def:completion-count}
For a state $\omega\in\Omega_m$ and an integer $t\ge0$, define
\[
F_t(\omega)
=
\text{the number of directed walks of length $t$ in $G_m$
beginning at $\omega$}.
\]
\end{definition}

These quantities satisfy
\[
F_0(\omega)=1
\]
and, for $t\ge0$,
\[
F_{t+1}(\omega)
=
\sum_{\omega\to\eta}
F_t(\eta).
\]

Thus $F_t(\omega)$ counts exactly how many length-$t$ extension histories remain available from the state $\omega$.

For the sampling procedure, we regard the finite initial level \(C_{m+1,m}\) as part of the preprocessing. For each
\(\omega\in\Omega_m\), let
\[
\mathcal{I}_m(\omega)
=
\left\{
\pi\in C_{m+1,m}:
\lambda_m(\pi)=\omega
\right\}.
\]
Then
\[
|\mathcal{I}_m(\omega)|=v_m(\omega).
\]
Fix an ordering of each finite set \(\mathcal{I}_m(\omega)\), or, equivalently for sampling purposes, assume access to an exact uniform sampler on each such initial fiber. Since \(m\) is fixed, this is a finite preprocessing step independent of the target length \(n\). If complexity as a function of \(m\) is to be tracked, the cost of this initial-fiber preprocessing should be included separately.

\begin{prop}[Exact uniform generation]
\label{prop:uniform-generation}
Fix
\[
n=m+1+k.
\]
Assume that the values
\[
F_t(\omega),
\qquad
0\le t\le k,
\]
have been computed for all $\omega\in\Omega_m$. A uniformly random permutation from $\mathcal C_{n,m}$ may be generated as follows. First choose an initial state $\omega$ with probability
\[
\frac{
v_m(\omega)F_k(\omega)
}{
a_{n,m}
}.
\]
Conditional on choosing \(\omega\), choose uniformly from
\[
\mathcal{I}_m(\omega).
\]

Now suppose the current state is $\omega$ and $t\ge1$ insertions remain. Choose an outgoing edge
\[
\omega\longrightarrow\eta
\]
with probability
\[
\frac{F_{t-1}(\eta)}{F_t(\omega)}.
\]
Insert the next maximum at the active site represented by that edge, move to $\eta$, and continue until no insertions remain. The final permutation is uniformly distributed over $\mathcal C_{n,m}$.
\end{prop}

\begin{proof}
An initial state $\omega$ has
\[
v_m(\omega)
\]
possible initial permutations, and each such permutation has
\[
F_k(\omega)
\]
possible $k$-step extension histories. Thus
\[
v_m(\omega)F_k(\omega)
\]
is exactly the number of final permutations whose construction begins in state $\omega$.

Since
\[
a_{n,m}
=
\sum_{\omega\in\Omega_m}
v_m(\omega)F_k(\omega),
\]
the first sampling step chooses the correct fraction of the total class. Now suppose the current state is $\omega$ and $t$ steps remain. The total number of completions is
\[
F_t(\omega)
=
\sum_{\omega\to\eta}
F_{t-1}(\eta).
\]
The edge
\[
\omega\to\eta
\]
leads to exactly
\[
F_{t-1}(\eta)
\]
of those completions. Choosing the edge with probability
\[
\frac{F_{t-1}(\eta)}{F_t(\omega)}
\]
hence chooses among the possible future completions in the correct proportion. Iterating this argument shows that every full extension history has equal probability conditional on its initial permutation.

Finally, the unique-parent property implies that distinct initial permutations and insertion histories produce distinct final permutations. Hence every element of $\mathcal C_{n,m}$ is generated with probability
\[
\frac1{a_{n,m}}.
\]
\end{proof}

The algorithm has a simple interpretation: at each step, an outgoing edge is weighted by the number of valid permutations that can still be completed after taking that edge. The sampler, as a result, never has to generate an invalid permutation and reject it afterward.

\begin{remark}[Ranking and unranking]
\label{rem:ranking}
The same completion counts also support exact ranking and unranking.

Use the fixed ordering of each initial fiber
\(\mathcal{I}_m(\omega)\), and fix, for every state, an ordering of its outgoing active-site transitions. The quantity
\[
F_{t-1}(\eta)
\]
gives the size of the block of permutations beginning with the transition
\[
\omega\to\eta.
\]
A rank can thus be computed by summing the sizes of all earlier blocks along the unique insertion history of a permutation. Conversely, given a rank, these same block sizes determine which outgoing edge contains that rank. Repeating this step reconstructs the unique permutation having the prescribed rank.
\end{remark}

\subsection{Interpretation and scope}
\label{subsec:graph-scope}

The two graph models developed in this section encode different levels of the same problem. The path power $P_n^m$ is an \emph{object-level} graph. A permutation in $\mathcal C_{n,m}$ is an oriented Hamiltonian path in this graph subject to an additional order restriction:
\[
\operatorname{LIS}\le2.
\]
This formulation makes the bounded-jump geometry explicit and gives a natural interpretation of the localization theorem as confinement near the decreasing Hamiltonian traversal. The profile graph $G_m$ is instead a \emph{state-level} graph. It forgets most of the current permutation and remembers only the information needed to determine future maximum insertions. A directed walk in $G_m$ records the successive extension types encountered as the permutation grows.

Schematically,
\[
\text{permutation}
\quad\longleftrightarrow\quad
\text{Hamiltonian path in }P_n^m
\]
describes the object itself, whereas
\[
\text{successive active profiles}
\quad\longleftrightarrow\quad
\text{directed walk in }G_m
\]
describes its recursive construction.

The first viewpoint is primarily interpretive. Rephrasing the problem as a Hamiltonian-path problem does not, by itself, produce an easier enumeration: Hamiltonian paths are difficult to count in general, and the additional longest-increasing-subsequence condition remains global.

The second viewpoint is the source of the computational compression. For fixed $m$, the graph $G_m$ has only
\[
2^{m+1}-m-1
\]
states and
\[
O(m2^m)
\]
edges, independent of $n$. This finite graph is sufficient for exact counting, spectral analysis, uniform sampling, and ranking.

In particular, if the number of admissible length-$k$ extension histories grows on the scale
\[
\rho_m^k,
\]
then
\[
\log\rho_m
\]
may be viewed as the exponential growth rate per insertion of the finite-state extension process.

A classical bijection between \(123\)-avoiding permutations and Dyck paths, such as the one developed by Krattenthaler~\cite{Krattenthaler2001}, could provide additional insight if the bounded-adjacency condition translated into a comparably simple path restriction. The graph interpretations above require no such further bijection: they arise directly from the defining adjacency condition and from the finite-state insertion structure developed in the paper.

\section{Applications}

The transfer-matrix construction developed above is useful beyond enumeration. Its main feature is that the active profile records all of the information needed to determine which maximum insertions remain possible. Thus, two permutations with the same profile have the same possible future profile transitions. This allows the same state space to be used not only for counting, but also for optimization and exact random generation. Optimization under permutation-pattern restrictions has also been studied more broadly~\cite{BerendsohnKozmaOpler2024}.

\subsection{Weighted constrained sequencing}

One natural extension is to assign a cost to each jump in the permutation. Let \(c(u,v)\) denote the cost of placing \(v\) immediately after \(u\), and define
\[
W(\pi)=\sum_{i=1}^{n-1} c(\pi_i,\pi_{i+1}).
\]
We may then ask for a permutation in \(\mathcal{C}_{n,m}\) minimizing \(W(\pi)\). Equivalently, this asks for a minimum-cost oriented Hamiltonian path in \(P_n^m\) whose label order avoids \(123\). The same active profile used for counting contains enough information to update this cost. Suppose \(\omega=(d_1,\ldots,d_\ell)\) is the profile of a permutation of length \(n\). If \(n+1\) is inserted at the front, the only new edge is \((n+1,n-d_1)\), so the cost changes by
\[
\Delta_{n,0}(\omega)=c(n+1,n-d_1).
\]

If instead \(n+1\) is inserted between \(n-d_j\) and \(n-d_{j+1}\), the edge \((n-d_j,n-d_{j+1})\) is removed and replaced by \((n-d_j,n+1)\) and \((n+1,n-d_{j+1})\). Hence
\[
\Delta_{n,j}(\omega)
=
c(n-d_j,n+1)
+
c(n+1,n-d_{j+1})
-
c(n-d_j,n-d_{j+1}).
\]

The important point is that this change depends only on \(n\), the current profile, and the chosen active site. Nothing else about the permutation is needed.

For \(n\ge m+1\) and \(\omega\in\Omega_m\), let
\[
M_{n,m}(\omega)
=
\min
\left\{
W(\pi):
\pi\in C_{n,m},
\ \lambda_m(\pi)=\omega
\right\},
\]
with the convention that
\[
M_{n,m}(\omega)=+\infty
\]
if no permutation of length \(n\) has profile \(\omega\).

The dynamic program is initialized at length \(m+1\) by
\[
M_{m+1,m}(\omega)
=
\min
\left\{
W(\pi):
\pi\in C_{m+1,m},
\ \lambda_m(\pi)=\omega
\right\}.
\]
By Theorem~\ref{thm:realizable-profiles}, every state in \(\Omega_m\) is already realizable at this level, so each of these initial minima is taken over a nonempty finite set.

For \(n\ge m+1\), the subsequent values satisfy
\[
M_{n+1,m}(\eta)
=
\min_{\substack{
\omega\in\Omega_m\setminus\{\dagger\},\,
0\le j<|\omega|\\
\Phi_{m,j}(\omega)=\eta
}}
\left\{
M_{n,m}(\omega)+\Delta_{n,j}(\omega)
\right\}.
\]
Thus, after the initial layer and the profile transition graph have been constructed, each level requires one relaxation for each edge of \(G_m\). Since the graph has
\[
O(m2^m)
\]
transitions, the optimization through length \(n\) uses
\[
O(nm2^m)
\]
state-transition relaxations. This count treats evaluation of each required cost \(c(u,v)\) as a constant-time operation; if the cost function itself is computationally nontrivial, its evaluation cost should be included separately.

The same idea applies to objectives other than minimum cost. Transition weights can instead be used to maximize an objective, keep track of a statistic, or form a weighted generating function. In the particularly natural case \(c(u,v)=g(|u-v|)\), the cost update depends only on the deficits stored in the profile. The weighted transition rule is then independent of \(n\), just as in the unweighted counting problem.

\subsection{Exact generation and ranking}

Exact uniform generation, together with ranking and unranking, follows from the backward completion counts developed in Section~\ref{subsec:algorithms}. Proposition~\ref{prop:uniform-generation} gives the exact sampler, while Remark~\ref{rem:ranking} explains the corresponding ranking and unranking procedures. Hence, we do not need to repeat those constructions here. Together with the weighted optimization above, they illustrate that the active-profile state retains enough information for exact algorithmic operations on the class.

\section{Further Directions}

The finite-state description above depends on two special features of the problem: the local jump restriction and the particularly simple effect of inserting a new maximum into a \(123\)-avoiding permutation. Several natural questions arise from asking how much of the argument survives when either feature is changed.

\subsection{Minimality of the profile representation}

For fixed \(m\), the construction uses \(2^{m+1}-m-1\) profiles. These states arise naturally from the proof, but we have not shown that all of them are necessary.

It is possible that two different profiles always have identical future behavior and could thus be merged without changing the resulting counting sequence. This leads to the natural question of determining the smallest state space capable of enumerating \(\mathcal{C}_{n,m}\) exactly.

In particular, it would be interesting to determine whether the exponential dependence on \(m\) is unavoidable. A substantially smaller representation would reveal additional structure not captured by the present profiles. Conversely, an exponential lower bound would show that the size of the current state space reflects a genuine feature of the problem rather than merely the particular encoding used here.

\subsection{Other pattern restrictions}

The role of \(123\)-avoidance in the proof is unusually simple. When the new maximum \(n+1\) is inserted into a permutation, it creates a copy of \(123\) exactly when the portion preceding it contains an ascent. Thus, a legal insertion site must have a decreasing prefix.

This is what makes the pattern restriction interact so cleanly with the local jump condition. It is natural to ask what happens for another fixed pattern \(\tau\). Define
\[
\mathcal{C}_{n,m}(\tau)
=
\left\{
\pi\in\operatorname{Av}_n(\tau):
|\pi_{i+1}-\pi_i|\le m
\text{ for all }i
\right\}.
\]
A broader question is then to determine for which patterns \(\tau\) the class \(\mathcal{C}_{n,m}(\tau)\) admits a finite-state description for fixed \(m\).

The answer need not be the same for every pattern. What matters is how inserting the new maximum can create a forbidden occurrence. For \(123\), this depends only on whether the preceding prefix contains an ascent. For other patterns, relevant entries may occur on both sides of the new maximum, and more information may have to be retained in the state. A particularly natural test case is avoidance of \(12\cdots k\), or equivalently the condition \(\operatorname{LIS}(\pi)\le k-1\). When \(k=3\), a legal insertion before the new maximum requires the preceding prefix to have longest increasing subsequence of length at most \(1\), which simply means that it is decreasing. For \(k>3\), the corresponding prefix need only avoid \(12\cdots(k-1)\), and may, as a result, have much more complicated structure. Structural decompositions of permutations have also been used algorithmically in permutation-pattern matching. Guillemot and Marx introduced a width measure and obtained an algorithm whose running time is linear in the input length for every fixed pattern size~\cite{GuillemotMarx2014}. This suggests the more general algorithmic question of whether, for fixed \(m\) and \(k\), bounded-jump permutations satisfying \(\operatorname{LIS}(\pi)\le k-1\) can still be counted in \(f(m,k)n^{O(1)}\) time for some function \(f\).

\subsection{Beyond powers of paths}

The graph interpretation suggests another direction. The condition \(|\pi_{i+1}-\pi_i|\le m\) means that the permutation gives an oriented Hamiltonian path in the \(m\)-th power of a path, \(P_n^m\). The deletion argument used above relies heavily on the ordered structure of this graph. In particular, the earlier neighbors of a newly introduced large vertex occur in a controlled interval. Proper interval graphs admit an ordering with a closely related interval-neighborhood property: if
\[
v_i < v_j < v_k
\]
in such an ordering and \(v_iv_k\) is an edge, then both \(v_iv_j\)
and \(v_jv_k\) are edges~\cite{LO93}. This ordered neighborhood structure makes proper interval graphs a natural first test case. Given an ordering \(v_1<\cdots<v_n\) compatible with a proper interval representation, one may ask whether Hamiltonian paths whose index sequence avoids \(123\) can again be described using only a bounded amount of information near the newest vertex.

More broadly, it would be useful to determine which ordered graph classes admit a bounded profile under maximum insertion. Such a result would separate the features of the present proof that are specific to \(P_n^m\) from those that belong to a more general theory of pattern-constrained Hamiltonian paths on ordered graphs.

\subsection{The limit as the jump bound grows}

Theorem~\ref{thm:rho-bounds} only gives the upper estimate
\[
4-\rho_m = O(\frac{\log m}{m})
\]
and the asymptotic behavior
\[
\rho_m\longrightarrow 4.
\]

A natural next question is to determine the genuine rate at which $\rho_m$ approaches $4$. In particular, it would be valuable to know whether the logarithmic factor in the present upper bound is needed, and whether $4-\rho_m$ gives a sharper asymptotic formula.

The growing radius regime also leaves finer open questions to solve. Corollary~\ref{cor:growing-radius} shows that whenever $m = m(n) \to \infty$,
\[
a_{n,m(n)}^{1/n}\longrightarrow 4,
\]
so every diverging adjacency radius recovers the unrestricted Catalan growth rate on the exponential scale. However, this does not describe how many Catalan permutations survive the adjacency restriction at sub-exponential scales. 

As a result, one possible refinement is to study the ratio 
\[
\frac{a_{n,m(n)}}{C_n}
\]
for different growth regimes of $m(n)$. Depending on the rate at which $m(n)$ increases, this ratio may exhibit behavior that is invisible after taking $n\text{th}$ roots. More generally, one may ask for asymptotic estimates of $a_{n, m}$ that remain informative when both $m$ and $n$ vary. Such results would give a finer description of the transition between the fixed-$m$ finite-state regime and ordinary $123$-avoidance, beyond the exponential scale convergence established above. 

\nocite{*}
\bibliographystyle{alpha}
\bibliography{references}

\end{document}